\documentclass[dvipsnames,letterpaper,11pt]{article}
\usepackage[margin=1in]{geometry}
\usepackage[utf8]{inputenc}
\usepackage{graphicx,fullpage,paralist}
\usepackage{amsmath,amssymb,amsthm}
\usepackage{dsfont}
\usepackage{comment,hyperref}
\usepackage{booktabs}
\usepackage{subcaption}
\usepackage{wrapfig}
\usepackage[font={footnotesize}]{caption}
\usepackage{cleveref}
\usepackage{mathtools}
\usepackage{thmtools}
\usepackage{xstring}
\usepackage{enumitem}
\setlist[itemize]{nosep}
\setlist[enumerate]{nosep}
\usepackage{blkarray}
\usepackage[normalem]{ulem}
\usepackage{csquotes}
\usepackage{stmaryrd}
\usepackage[longnamesfirst]{natbib}
\usepackage{thm-restate}
\usepackage{tikz}
\usetikzlibrary{positioning,quotes}
\usepackage{xspace}
\usepackage{doi}

\newcommand{\PP}{\mathbb{P}}
\newcommand{\EE}{\mathbb{E}}
\newcommand{\RR}{\mathbb{R}}
\newcommand{\NN}{\mathbb{N}}

\newcommand{\calP}{\mathcal{P}}

\newcommand{\stay}{\mathrm{S}}
\newcommand{\move}{\mathrm{M}}
\newcommand{\meet}{\mathrm{R}}
\newcommand{\first}{\varphi}

\newcommand{\AW}{\mathrm{AW}}
\newcommand{\SAW}{\mathrm{SAW}}
\newcommand{\CFKt}{First-Same-or-Different\xspace}
\newcommand{\CFK}{\mathrm{FSD}}
\newcommand{\SCFK}{\mathrm{SFSD}}
\newcommand{\Uni}{\textsc{Uni}}
\newcommand{\Vis}{\textsc{Cli}}

\newtheorem{theorem}{Theorem}
\newtheorem{lemma}{Lemma}
\newtheorem{corollary}{Corollary}

\usepackage[textsize=tiny,textwidth=2cm]{todonotes}

\crefformat{equation}{(#2#1#3)}

\newcommand{\ie}{i.e.,\ }

\usepackage{multirow}
\usepackage[ruled]{algorithm2e}
\usepackage{xcolor}

\newlength{\algofontsize}
\hypersetup{
	colorlinks,
	linkcolor={red!50!black},
	citecolor={blue!50!black},
    urlcolor={black}
}

\def\equationautorefname~#1\null{(#1)\null}

\usepackage{etoolbox}
\makeatletter
\patchcmd{\hyper@makecurrent}{%
    \ifx\Hy@param\Hy@chapterstring
        \let\Hy@param\Hy@chapapp
    \fi
}{%
    \iftoggle{inappendix}{%
        \@checkappendixparam{chapter}%
        \@checkappendixparam{section}%
        \@checkappendixparam{subsection}%
        \@checkappendixparam{subsubsection}%
        \@checkappendixparam{paragraph}%
        \@checkappendixparam{subparagraph}%
    }{}%
}{}{\errmessage{failed to patch}}

\newcommand*{\@checkappendixparam}[1]{%
    \def\@checkappendixparamtmp{#1}%
    \ifx\Hy@param\@checkappendixparamtmp
        \let\Hy@param\Hy@appendixstring
    \fi
}
\makeatother

\newtoggle{inappendix}
\togglefalse{inappendix}

\apptocmd{\appendix}{\toggletrue{inappendix}}{}{\errmessage{failed to patch}}

\begin{document}
	
\title{Faster Symmetric Rendezvous on Four or More Locations}

\author{Javier Cembrano\thanks{Department of Industrial Engineering, Universidad de Chile.}
\and Felix Fischer\thanks{School of Mathematical Sciences, Queen Mary University of London.}		
\and Max Klimm\thanks{Institute for Mathematics, Technische Universität Berlin.}
}

\date{}

\maketitle

\begin{abstract}
In the symmetric rendezvous problem, two players follow the same (randomized) strategy to visit one of~$n$ locations in each time step~$t=0,1,2,\dots$. Their goal is to minimize the expected time until they visit the same location and thus meet. 
A canonical strategy due to Anderson and Weber (\textit{J. Appl. Probab.}, 1990) is known to be optimal for $n=2$ and $n=3$ (Weber, \textit{Math. Oper. Res.} 2012), but whether it remains optimal for larger values of $n$ has been an open question since 1990. We show that it does not remain optimal: for any finite $n\geq 4$, we construct an explicit symmetric strategy that achieves a strictly smaller expected meeting time than the Anderson--Weber strategy. 

In the Anderson--Weber strategy players stay at a dedicated home location for $n-1$ steps with a certain probability $\theta$ and with the remaining probability tour all non-home locations in a random order. Our improving strategy introduces carefully chosen correlations between consecutive tours of the non-home locations. The construction is uniform in $n$ and is guided by a graph-theoretic view in which tours correspond to permutations and meetings to edges in the complement of the derangement graph. By exploiting the clique structure of this graph we obtain a correlated strategy that improves on the Anderson--Weber strategy. For $n=4$, we give an exact expression for the improvement; for any $n\geq 5$, we obtain a lower bound on the expected improvement of 
$\frac{483(1-\theta)^6}{(n-1)^8}$, where $\theta \in [0,1)$ is the probability of staying at the home location. The graph-theoretic framework we introduce may be useful more widely in the design and analysis of correlated strategies for rendezvous.
\end{abstract}

\newpage

\section{Introduction}

In 1976, Steve Alpern gave a talk at the Vienna Institute for Advanced Studies where he proposed the following innocent-looking problem~\citep[cf.][]{Alpern02}:
\begin{quote}
In each of two rooms, there are~$n$ telephones randomly strewn about. They are connected in a pairwise fashion by~$n$ wires. At discrete times~$t=0,1,2,\dots$ players in each room pick up a phone and say ‘hello.’ They wish to minimize the time~$t$ when they first pick up paired phones and can communicate.
\end{quote}

This problem, which can alternatively be thought of in terms of two individuals trying to meet in one of~$n$ physical locations such as rooms in a house or shops in a mall, has become known as the rendezvous problem on discrete locations. 
Since there is nothing in the description of the problem that distinguishes the telephones or locations, it is common to assume that they are given a labeling for each player that is chosen uniformly at random and independently from the labeling used by the other player. An optimal strategy then is one that minimizes the expected meeting time, where the expectation is taken over the random labelings and any randomness in players' strategies. 
Whether a distinction can be made between the players leads to two natural versions of the problem. The asymmetric version, where players are allowed to use distinct strategies, is straightforward. In the optimal pair of strategies, called wait-for-mommy, one player stays in the same location throughout while the other tours all locations in random order. This leads to an expected meeting time of~$(n-1)/2$. 
In the symmetric version, both players must use the same strategy. A strategy that like wait-for-mommy assigns distinct roles to the players hence cannot be used, and any deterministic strategy will with significant probability chase itself forever and therefore has an infinite expected meeting time. 

It is natural to exploit what we know about the asymmetric problem, while using randomness to break symmetry. 
As our objective will be to prove upper bounds on the optimal meeting time, we may assume that locations at time~$0$ are chosen uniformly at random; we will call the location a player visits at time~$0$ her \emph{home location}. We can then focus on strategies that start at time~$1$ and have access to a home location guaranteed to be distinct from the home location of the other player.
\citet{AW90} proposed a strategy, now called the Anderson--Weber strategy or $\AW$ for short, in which players repeatedly choose between waiting and touring: with some probability~$\theta$ a player stays in her home location for~$n-1$ steps, while with probability~$(1-\theta)$ she tours the remaining~$n-1$ locations in random order. For~$n=2$ and~$\theta=1/2$, this strategy just chooses a random location in each step, and it is not difficult to show that this is optimal~\citep{AW90}. \citeauthor{AW90} also claimed that their strategy is optimal for~$n=3$, but their proof turned out to be flawed. Optimality for~$n=3$ was finally shown by \citet{Weber12}. The difficult proof proceeds by showing that the matrix of meeting times for an easier problem, truncated after~$k$ steps for any~$k$, can be rounded down to a matrix which is positive semidefinite and for which $\AW$ is optimal.

The simple description of the symmetric rendezvous problem belies significant difficulty, which not only concerns the search for good strategies but also some very basic properties.
\citeauthor{AW90} conjectured for example that the optimal meeting time must increase with the number of locations, as it does in the asymmetric version, but this conjecture is still open.
It also seems clear that the symmetric version is more difficult than the asymmetric one, but this was only shown very recently and only for a difference in meeting times of~$2^{-36}n$~\citep{BPSW23a}.
Better lower bounds on the meeting time only exist for~$n\leq 5$, where computational techniques can be used~\citep{Fan09}, and for~$n\to\infty$~\citep{DHMR16}.
Optimal strategies for $n\geq 4$, and the conjecture of \citeauthor{AW90} that $\AW$ is optimal for~$n\to\infty$, currently seem out of reach.

\subsection{Our Contribution}

We give a strategy, which we call \CFKt and abbreviate~$\CFK$, that has a smaller expected meeting time than $\AW$ for any~$n\geq 4$. 
\citet{AW90} had conjectured, albeit without justification, that an improvement would be possible for~$n=4$. \citet{Fan09}, on the other hand, conjectured optimality of $\AW$ for all~$n$ and based this on computational evidence. An unpublished manuscript of \citet{Weber09} finally gave concrete evidence that an improvement is possible for $n=4$.  \citeauthor{Weber09} modifies $\AW$ over blocks of~$12$ consecutive steps, corresponding to four blocks of~$n-1=3$ steps each in which a player tours; in the first two blocks locations are visited in random order as in $\AW$, but orders in the latter two blocks are correlated. This strategy was derived from a negative eigenvalue of a matrix of meeting times, using intuition gained for~$n=3$ on the relationship between positive semidefiniteness and optimality of $\AW$.

\citeauthor{Weber09} suggests that the improvement over $\AW$ can be seen most easily by calculating the meeting time for~$1585^2$ pairs of sequences of locations, and weighting each meeting time by the probability of the corresponding pair. These calculations are not made explicit but are feasible with the help of a computer. When it comes to the suboptimality of $\AW$ for~$n=4$ this argument is convincing enough, but it does not appear to generalize to larger values of~$n$. Computing the meeting time explicitly for an equally complicated strategy when~$n>4$ is out of the question. More importantly, the computations underlying \citeauthor{Weber09}'s choice of improving strategy also become infeasible for~$n>4$, leaving us without a candidate strategy.

In $\AW$, and $\AW$-like strategies like that of \citeauthor{Weber09}, players choose a permutation in each round and meet if the two permutations have at least one fixed point, \ie if they are \emph{not} derangements.\footnote{Here and in the following we call two permutations derangements if they do not share a position with the same entry. A single permutation is called a derangement if it is a derangement of the identity.} It is therefore a bit surprising that research on the rendezvous problem has not used what is known about derangements beyond the classic estimate that a random permutation is a derangement with probability approaching~$1/e$ as~$n\to\infty$. We will use intuition, and somewhat deeper results, on the combinatorics of derangements to construct and analyze a strategy that improves over $\AW$. 
Like \citeauthor{Weber09}, we take $\AW$ as a starting point, operate in rounds of length~$n-1$ consisting either of repeated visits to the home location or of a tour of the remaining locations, and introduce correlation among consecutive permutations. 
Within each block of three rounds in which a player moves throughout, the first two tours are sampled independently and uniformly. The third is sampled uniformly subject to starting at the same location as the first two if their first locations agree, and at a location different from both otherwise.
This is well-defined for all $n\geq 4$, and we will see that it leads to an improvement over $\AW$ for all of these values. 
Denoting Anderson--Weber's and our strategy for a specific value of the parameter $\theta$ by $\AW(\theta)$ and $\CFK(\theta)$, respectively, our main result is the following.
\begin{theorem}[\Cref{thm:meeting-time} (informal)]
    Let~$n\in \NN$ and~$\theta\in [0,1)$. The expected meeting time of~$\CFK(\theta)$ is smaller than the expected meeting time of~$\AW(\theta)$ by~$\frac{8(1-\theta)^6}{81(8-(3\theta^2-2\theta+1)^3)}$ if~$n=4$, and by at least~$\frac{483(1-\theta)^6}{(n-1)^8}$ if~$n\geq 5$.
\end{theorem}

For the values of~$\theta$ that are optimal for~$\AW(\theta)$, we obtain an improvement of approximately~$0.00125$ for~$n=4$ and of at least~$0.00086$ for~$n=5$.
This yields respective improvements of approximately~$0.000938$ and at least~$0.00069$ compared to the expected meeting times of $2.5685$ and $3.3793$ of $\AW$ for the original problem where players meet at time $0$ with probability $1/n$.\footnote{These values differ by~$1$ from those given by \citet{AW90}, who count time steps starting from~$1$.}

To see why the type of correlation we use might be helpful, we can view round-based strategies like $\AW$ as being played on a graph, where vertices represent permutations and two vertices are connected by an edge if they are \emph{not} derangements.
This is the complement of the so-called \emph{derangement graph} studied in combinatorics \citep[e.g.][]{MeagherRS21,Renteln07,Rasmussen94}.
When a player tours the other locations, $\AW$ ignores the structure of the graph and repeatedly chooses a vertex uniformly at random. Our strategy chooses three vertices that either belong to a large clique or to three distinct large cliques; the choice between these two options is made randomly with a probability that depends only on~$n$. This can be seen as replicating on cliques of permutations what $\AW$ does on locations, and one may hope that it improves over $\AW$ in the same way in which $\AW$ improves over a strategy that in each step chooses a random location. Whether it actually provides an improvement depends on the clique structure, and showing that it does so for our particular choice of strategy requires a fairly intricate analysis. The analysis is complicated further by the fact that the permutations used by the two players are not of the same set of locations but of sets that differ by one element, owing to the fact that players do not visit their home location when moving. Moreover, for the expected meeting time, it not only matters whether two permutations have a fixed point but also where in the permutations the first fixed point occurs.

\citet{Fan09} and \citet{Alpern13} contemplate that a good symmetric strategy for rendezvous could be published in a survival guide that one could consult if one needed to meet in an unknown environment. While its analysis is fairly involved, our strategy is simple enough to be published in such a guide. The only change compared to~$\AW$ affects rounds divisible by three, when we have moved for the previous two rounds and decide to move again; if in the previous two rounds we started in the same location, we again start in that location; if we started in distinct locations, we start in a location that is distinct from both; if the tour we have sampled for the current round does not satisfy this criterion, we sample it again until it does.

\subsection{Structure of the Paper}

Our analysis of~$\CFK$ relies heavily on the following intermediate result, stated as \Cref{thm:meeting-prob} in \Cref{subsec:meeting-prob}: for~$n\geq 5$,~$\CFK$ has a higher probability of meeting than~$\AW$ by the end of round~$r$ for any~$r\geq 3$; for~$n=4$ the probabilities are equal.
To gain intuition and illustrate the key ideas of the proof, we prove a simplified version of this result in \Cref{sec:warm-up}, concerning variants of~$\AW$ and~$\CFK$ that visit all~$n$ locations in each moving round.
We then prove \Cref{thm:meeting-prob} by expressing the probability of meeting within the first three rounds, under $\CFK$ and under the condition that both players move in all of these rounds, in terms of the proportion of \textit{shifted derangements}. Shifted derangements generalize the notion of derangements to pairs of permutations of distinct sets; a formal definition, and a bound on their proportion among the set of all permutations that is used in the proof of the theorem, are given in \Cref{subsec:derangements}.

Our main result, concerning the expected meeting time under $\CFK$, is shown in \Cref{subsec:meeting-time}.
The key ingredient of the proof, along with \Cref{thm:meeting-prob}, is \Cref{lem:diff-third-round}, which gives an improvement in expected meeting time under the condition that the players move in the first three rounds and meet in the third round.
To prove the lemma we further condition on the respective first locations the players visit in the third round and compute, for both $\AW$ and $\CFK$, the conditional expectations and probabilities.
The conditional expectations are the same for both strategies and are given in \Cref{lem:cond-expectations}; the key ingredient for their computation is an explicit expression for the expected index of the first fixed point between a permutation taken uniformly at random from~$\calP(\{1,\ldots,n\})$ and a permutation taken uniformly at random from~$\calP(\{1+b,\ldots,n+b\})$, conditional on them having at least one fixed point (\Cref{lem:exp-first-fixed-point}). 
Regarding the conditional probabilities we show that, under the condition that players move for all of the first three rounds and meet for the first time in the third round, $\CFK$ has a larger bias towards the following: (i) the players meeting in the first step of the third round, which leads to the best possible meeting time under the given conditions; and (ii) both players visiting the home location of the other player in the first step of the third round, which leads to an earlier meeting time later in the round compared to the case where players visit distinct locations in the first step that are not home locations.
Computing the conditional probabilities for~$\CFK$ is nontrivial due to the correlations between different rounds; we first compute the joint probabilities with which the players fail to meet in the first two rounds and visit a certain pair of locations in the first step of the third round (\Cref{lem:cond-probabilities}) and then apply Bayes' rule.

The fact that \Cref{thm:meeting-time} follows from \Cref{thm:meeting-prob} and \Cref{lem:diff-third-round} is intuitive but nontrivial.
To prove it, we use that both $\AW$ and $\CFK$ restart every~$3(n-1)$ steps to write the unconditional expectations in terms of expectations conditioned on either (i)~meeting within the first two rounds or (ii)~meeting in the third round. In the first case conditional expectations are the same under both strategies, and we bound their difference in the second case. For $n=4$ we exploit the fact that the meeting probabilities within each round are the same under both strategies. For~$n\geq 5$ we show that, conditioned on meeting in the third round, (i)~the expected meeting time is lower when both players move for all three rounds than in the other cases and (ii) this event has a higher probability under~$\CFK$ than under~$\AW$ by \Cref{thm:meeting-prob}.

\subsection{Further Related Work}

Rendezvous problems have been discussed informally many times, for example by \citet{Schelling60} for two parachutists who have landed in a field and by \citet{Mosteller65} for two strangers wanting to meet in New York City. The particular problem we study here was first stated formally by Alpern in 1976~\citep[cf.][]{Alpern02}, as a more tractable discrete version of the astronaut problem, where players try to meet on the surface of a sphere. No significant progress seems to have been made on the astronaut problem, and the intuition about the relative difficulty of these two problems was most certainly correct. But as we have already discussed, discrete rendezvous problems are anything but tractable. 

Another notorious example places the two players on the real line, at distance two and without a common orientation~\citep{Alpern95}. While this problem was formulated as a continuous one, it turns out to be discrete since strategies where players move at maximum speed and change direction only at unit time steps dominate all other strategies~\citep{HDVZ08}. The asymmetric version is again not very difficult, and has an optimal meeting time of~$3.25$~\citep{AlGa95}. It is worth noting that in the optimal pair of strategies one player remains more stationary than the other, but not completely stationary. The optimal meeting time for the symmetric version lies in the interval~$(4.1520,4.2574)$~\citep{HDVZ08}; both the upper and the lower bound were obtained using semidefinite programming. 
The problem has also been studied for the case where the initial distance is unknown and the goal is to optimize the competitive ratio between the meeting time and half the initial distance~\citep{Baston98}. The competitive ratio is at most~$11.028$ for asymmetric strategies~\citep{AlBe00}, which has been conjectured to be optimal, and at most~$13.926$ for symmetric strategies~\citep{KlimmSST22}. 
\citet{GGMK26} recently considered symmetric rendezvous on the integer line when a player's next move can only depend on its displacement from its own starting position.
The problem where only one player moves, known as the cow path problem, is somewhat more tractable. Here optimal deterministic and randomized strategies are known, with competitive ratios~$9$~\citep{BCGR93} and~$4.591$~\citep{KRT96a}, respectively. 

The standard objective in rendezvous search is minimizing the expected meeting time, but \citet{DHMR16} have studied our problem with the goal of maximizing the probability of meeting after at most~$k$ steps for some~$k$, in the limit as~$n\to\infty$. It turns out that $\AW$ is optimal in this respect when~$k\leq n$ but far from optimal when~$k\geq 4n$; indeed, while $\AW$ fails to meet with constant probability after any fixed number of steps, there exists a strategy that meets within $4n$ steps with probability $1-o(1)$ as n grows. The analysis of meeting probabilities has also led to a lower bound on the expected meeting time of~$0.6389n$, again as~$n$ grows; the best known strategy for this case is $\AW$ with a meeting time of around~$0.8289n$.

\citet{ABE99} have considered a generalization where the game is played on a graph and players can only move to adjacent vertices.\footnote{Note that this is a different graph from the one we have considered earlier, where players meet when they visit adjacent vertices.} The absence of some edges makes the problem harder by restricting travel but may provide opportunities for coordination. It clearly offers an advantage if the graph is \emph{not} vertex-transitive, by allowing players to restrict attention to a subset of the vertices, but even for vertex-transitive graphs it can make the problem significantly easier at least with respect to meeting probabilities~\citep{DHMR16}.

A sizeable literature exists on continuous problems, and problems on the line in particular. \citet{Alpern11} provides a detailed overview of the area with many open questions.

\section{Preliminaries}\label{sec:prelims}

Let~$\NN$ denote the positive integers and~$\NN_0\coloneqq\NN\cup\{0\}$. For~$n\in\NN$, let~$[n]\coloneqq\{1,2,\dots,n\}$ be the set of integers from~$1$ to~$n$; we use the convention $[0]\coloneqq \emptyset$. For a set~$S$, let~$\calP(S)$ be the set of permutations of~$S$, understood as bijections~$\pi\colon[|S|]\to S$ mapping indices in~$\{1,\ldots,|S|\}$ to distinct elements of~$S$.

A \textit{strategy} for the symmetric rendezvous problem on~$n$ locations is a distribution over infinite sequences of elements of~$[n]$. We will often also use the term strategy to refer to a family of such distributions, one for each value of~$n$, and may omit the dependence on~$n$ when its value is clear from context.
We will denote sequences by functions~$s\colon\NN_0\to[n]$ and strategies by distributions~$\sigma$ over such functions. 
We will assume for simplicity that both players apply the same strategy according to their own labeling of the~$n$ locations, and that the bijection between their two labelings is drawn uniformly at random.
The \textit{meeting time} of a strategy~$\sigma$ is then
\(
    T_\sigma\coloneqq \min\{t\in \NN_0\colon \pi(s^1(t)) = s^2(t)\},
\)
where~$s^1$ and~$s^2$ are sequences drawn independently from~$\sigma$ and~$\pi$ is a permutation drawn uniformly at random from~$\calP([n])$.\footnote{We use the convention $\min \emptyset = \infty$.}

\subsection{Home Location and Round-based Strategies}\label{subsec:round-based}

We only consider strategies that choose the first location uniformly at random. Let~$\sigma$ be such a strategy, and let~$t_\sigma$ denote a random variable distributed as \(T_\sigma\) conditional on \(T_\sigma>0\).
Then
\[ \PP[T_\sigma\leq t^*] = \frac{1}{n} + \frac{n-1}{n}\PP[t_\sigma\leq t^*] \quad\text{for all~$t^*\in \NN$, and} \qquad \EE[T_\sigma] = \frac{n-1}{n}\EE[t_\sigma]. \]
 From now on, we will analyze the expected meeting times of strategies in terms of~$t_{\sigma}$. For simplicity, we will further assume a \textit{universal labeling} of the locations under which player~$i\in\{1,2\}$ starts at location~$i$ at time~$0$ and then visits a (random) sequence~$s^i\colon\NN\to[n]$ of locations. We will refer to the location~$i$ at which player~$i$ starts as her \emph{home location}. The strategies we will consider are symmetric because they treat non-home labels symmetrically and use independent private randomness.

We will focus on \textit{round-based strategies} with rounds of length~$\ell$ for some~$\ell\in\NN$, or~\emph{$\ell$-strategies} for short.  The time steps between~$(r-1)\ell+1$ and~$r\ell$ for~$r\in\NN$ will be called the \emph{$r$th round}. Players decide at the beginning of a round whether they stay at their home location for the entire round or move, possibly in a random way, across locations; this process is repeated until rendezvous. More formally, player~$i\in\{1,2\}$ chooses a sequence~$\chi^i_1,\chi^i_2,\ldots$, where~$\chi^i_r\in\{\stay,\move\}$ for each~$r\in\NN$, and a sequence~$\pi^i_1,\pi^i_2,\ldots$, where~$\pi^i_r\colon [\ell]\to [n]$ is a sequence of length~$\ell$ for each~$r\in\NN$.
If~$\chi^i_r=\stay$, then~$s^i((r-1)\ell+j)=i$ for all~$j\in [\ell]$, \ie player~$i$ \emph{stays} at her home location during the~$r$th round. 
If~$\chi^i_r=\move$, then~$s^i((r-1)\ell+j)=\pi^i_{r}(j)$ for every~$j\in [\ell]$, \ie player~$i$ \textit{moves} across locations in an order given by~$\pi^i_{r}$.
A particular round-based strategy is completely defined by the round length~$\ell$ and the distributions from which the sequences~$\chi^i_1,\chi^i_2,\ldots$ and~$\pi^i_1,\pi^i_2,\ldots$ are drawn.

\citet{AW90} proposed an~$(n-1)$-strategy with parameter~$\theta\in [0,1]$, which we will denote~$\AW(\theta)$, where 
\begin{enumerate}[label=(\roman*)]
    \item $\chi^i_r=\stay$ with probability~$\theta$ and~$\chi^i_r=\move$ with probability~$1-\theta$ for each~$i\in \{1,2\}$ and~$r\in\NN$, independently of all other rounds; and
    \item $\pi^i_r$ is a permutation taken uniformly at random from~$\calP([n]\setminus \{i\})$ for each~$i\in\{1,2\}$ and~$r\in\NN$, independently of all other rounds.
\end{enumerate}
With the goal of reducing the expected meeting time we define a different~$(n-1)$-strategy, which we will call \emph{\CFKt} with parameter~$\theta\in [0,1]$ and denote~$\CFK(\theta)$, where 
\begin{enumerate}[label=(\roman*)]
    \item $\chi^i_r=\stay$ with probability~$\theta$ and~$\chi^i_r=\move$ with probability~$1-\theta$ for each player~$i\in \{1,2\}$ and round~$r\in \NN$, independently of all other rounds; and
    \item for each~$i\in \{1,2\}$,~$\pi^i_r$ is a permutation taken uniformly at random from~$\calP([n]\setminus \{i\})$ for each round~$r\in \NN$ with~$r/3\notin \NN$, independently of previous rounds, and~$\pi^i_{r}$ is a permutation taken uniformly at random from the set \vspace{-1ex}
    \[
        \begin{array}{@{}ll}
            \{\pi\in \calP([n]\setminus \{i\}): \pi(1)\!=\!\pi^i_{r-2}(1)\} & \text{if } \chi^i_{r-2}\!=\!\chi^i_{r-1}\!=\!\chi^i_r\!=\!\move ,\ \pi^i_{r-2}(1)=\pi^i_{r-1}(1),\\
            \{\pi\in \calP([n]\setminus \{i\}): \pi(1)\!\notin\! \{\pi^i_{r-2}(1),\pi^i_{r-1}(1)\}\} & \text{if } \chi^i_{r-2}\!=\!\chi^i_{r-1}\!=\!\chi^i_r\!=\!\move ,\ \pi^i_{r-2}(1)\neq \pi^i_{r-1}(1),\\
            \calP([n]\setminus \{i\}) & \text{otherwise,}
        \end{array}
    \] \par\vspace*{-1ex}
    for each~$r\in \NN$ with~$r/3\in \NN$.
\end{enumerate}
For both strategies, in each round, the players stay at their home location with probability~$\theta$ and tour all non-home locations with the remaining probability~$1-\theta$.
In~$\AW(\theta)$, tours are taken independently and uniformly at random for each round where a player moves. 
$\CFK(\theta)$ differs from $\AW(\theta)$ only in rounds that are multiples of three, and only if a player moves in this and the previous two rounds.
In this case, if the previous two permutations start with the same location,~$\CFK(\theta)$ chooses the third permutation uniformly at random from those also starting with this common location; if the previous two permutations start with distinct locations, $\CFK(\theta)$ chooses the third permutation uniformly at random from those starting with a location that differs from both. 
We note that this is well defined when the number of locations is at least four.

\subsection{Shifted Derangements}\label{subsec:derangements}

For a permutation~$\pi$, call~$i\in [n]$ such that~$\pi(i)=i$ a \emph{fixed point} of~$\pi$, and call~$\pi$ a \emph{derangement} if it does not have a fixed point.
For two permutations~$\pi^1$ and~$\pi^2$, call~$i\in [n]$ such that~$\pi^1(i)=\pi^2(i)$ a \emph{fixed point} between~$\pi^1$ and~$\pi^2$, and call~$\pi^1$ and~$\pi^2$ \emph{derangements} if there is no fixed point between them.

When both players move in a round, each player~$i$ plays a random permutation~$\pi^i$ of the locations $[n] \setminus \{i\}$, i.e., the first player plays a random permutation of the locations $\{2,\dots,n\}$, the second player plays a random permutation of the locations $\{1\} \cup \{3,\dots,n\}$, and we will be interested in the probability that they coincide at some position, i.e., that $\pi^1(j) = \pi^2(j)$ for some $j \in [n-1]$. As no meeting at location 1 is possible, this is the same as the probability that two random permutations of $\{2,\dots,n\}$ and $\{3,\dots,n+1\}$ agree at a position which is, in turn, equal to the probability that two random permutations of $\{1,\dots,n-1\}$ and $\{2,\dots,n\}$ agree at a position. This probability is independent of the first permutation, so we may assume that it is the identity on the first $n-1$ elements.
So the analysis of the $\AW$ strategy naturally involves the probability that a permutation of the elements $\{1+k,\dots,n+k\}$ is a derangement for $k=1$. 
As we will see, the analysis of $\CFK$ also involves that probability for larger values of $k > 1$ because fixing the first position may remove further possibilities of meeting from the remaining round.

Thus, explicit and approximate values for this probability are needed.
To obtain these, we first use a recursion dating back to \citet{Euler1753} who defined a difference table with entries~$(d^k_{n})$ for~$n \in \NN_0$ and~$k \in \{0,\dots,n\}$ given by the following recurrence relation:
\[\begin{aligned}
d_n^n &\coloneqq n! &&\text{ for all } n \in \NN_0, \\
d_n^k &\coloneqq d_n^{k+1} - d_{n-1}^k &&\text{ for all } n \in \NN, k \in \{0,\dots,n-1\}.
\end{aligned}\]
Some initial rows and columns of this table are shown in \Cref{tab:euler}. 
\citeauthor{Euler1753}'s motivation was to calculate the number of derangements, and he proved that this number is equal to~$d_n^0$. The following lemma is a slight generalization of this result, which is already implicit in \citeauthor{Euler1753}'s work.

\begin{restatable}{lemma}{lemShiftedDerangements}\label{lem:shifted-derangements}
For all~$n \in \NN_0$ and~$k \in \{0,\dots,n\}$, the entry~$d_n^k$ is equal to the number of permutations~$\pi \in \mathcal{P}(\{1+k,\dots,n+k\})$ with~$\pi(i) \neq i$ for all $i\in [n]$.
\end{restatable}

\Cref{lem:shifted-derangements} was shown in a slightly different context by \citet{Feinsilver2011}.
For the sake of completeness, we include a proof in our context and notation in Appendix~\ref{app:lem:shifted-derangements}.

\citeauthor{Feinsilver2011} also showed that
\[\begin{aligned}
d_n^k = \sum_{j=k}^n (-1)^{n-j} \binom{n-k}{n-j} j! = \sum_{j = 0}^{n-k} (-1)^j \binom{n-k}{j} (n-j)! ,
\end{aligned}\]
which yields in particular $\smash{d_n^0 = n! \sum_{j=0}^{n} \frac{(-1)^j}{j!}}$. Noting that $\smash{\sum_{j=0}^\infty \frac{(-1)^j}{j!}} = 1/e$, this shows that the number of derangements can be approximated by $n!/e$. Making the errors in this approximation explicit and using the recursive nature of the values $d_n^k$, we obtain the following bounds.

\begin{restatable}{lemma}{lemBounds}\label{lem:bounds}
The following equalities hold:
\begin{enumerate}
\item $d_n^0 = \lfloor \frac{n!}{e} + \epsilon\rfloor$ for any~$n\in \NN$ and~$\epsilon \in \bigl[\frac{1}{3},\frac{1}{2}\bigr]$;
\item $d_n^1 = \lfloor \frac{n! + (n-1)!}{e} + \epsilon\rfloor$ for any~$n\in \NN$ and~$\epsilon \in \bigl[\frac{1}{3},\frac{7}{8}\bigr]$;
\item $d_n^2 = \lfloor \frac{n! + 2(n-1)! + (n-2)!}{e} + \epsilon\rfloor$ for any~$n\in \NN$ with~$n\geq 2$ and~$\epsilon \in \bigl[\frac{1}{3},\frac{7}{8}\bigr]$.
\end{enumerate}
\end{restatable}
We defer the proof to Appendix~\ref{app:lem:bounds}. The calculations for the approximation of~$d_n^0$ were already given by \citet{Hassani2004cycles}.

\begin{wraptable}[9]{r}{6cm}
\centering
\vspace{-.2cm}
\begin{tabular}{rrrrrrrr}
\toprule
&& \multicolumn{6}{c}{$k$} \\
&& 0 & 1 & 2 & 3 & 4 & 5 \\
\midrule
\multirow{6}{*}{$n$} & \multicolumn{1}{r|}{0} & 1\\
& \multicolumn{1}{r|}{1} & 0 & 1\\
& \multicolumn{1}{r|}{2} & 1 & 1 & 2 \\
& \multicolumn{1}{r|}{3} & 2 & 3 & 4 & 6 \\
& \multicolumn{1}{r|}{4} & 9 & 11 & 14 & 18 & 24 \\
& \multicolumn{1}{r|}{5} & 44 & 53 & 64 & 78 & 96 & \!\!120\\

\bottomrule
\end{tabular}
\caption{Euler's difference table \label{tab:euler}
}
\end{wraptable}

For~$n\in \NN_0$ and~$k\in \{0,\ldots,n\}$, we denote by  $\hat{d}^k_n\coloneqq d^k_n/n!$ the \emph{fraction}~of permutations in~$\calP(\{1+k,\ldots,n+k\})$ that do not have a fixed point.

\section{Warm-Up: Correlating Consecutive Permutations}\label{sec:warm-up}

To gain intuition, and illustrate the key ideas behind our improvement over~$\AW$, let us first consider simplified strategies that treat the home location like any other location when moving. We will focus on meeting probabilities for this section.

Let the Simplified Anderson--Weber strategy with parameter~$\theta$, or~$\SAW(\theta)$ for short, 
be the~$n$-strategy where 
\begin{enumerate}[label=(\roman*)]
    \item $\chi^i_r=\stay$ with probability~$\theta$ and~$\chi^i_r=\move$ with probability~$1-\theta$ for each player~$i\in \{1,2\}$ and round~$r\in \NN$, independently of all other rounds; and
    \item $\pi^i_r$ is a permutation taken uniformly at random from~$\calP([n])$ for each player~$i\in \{1,2\}$ and round~$r\in \NN$, independently of all other rounds.
\end{enumerate}
Let the Simplified \CFKt strategy with parameter~$\theta$, or~$\SCFK(\theta)$, be the~$n$-strategy where 
\begin{enumerate}[label=(\roman*)]
    \item $\chi^i_r=\stay$ with probability~$\theta$ and~$\chi^i_r=\move$ with probability~$1-\theta$ for each player~$i\in \{1,2\}$ and round~$r\in \NN$, independently of all other rounds; and
    \item for each~$i\in \{1,2\}$,~$\pi^i_r$ is a permutation taken uniformly at random from~$\calP([n])$ for each round~$r\in \NN$ with~$r/3\notin \NN$, independently of previous rounds, and~$\pi^i_r$ is a permutation taken uniformly at random from the set \vspace*{-1ex}
    \[
        \begin{array}{@{}ll}
            \{\pi\in \calP([n]): \pi(1)=\pi^i_{r-2}(1)\} & \text{if } \chi^i_{r-2}=\chi^i_{r-1}=\chi^i_r=\move \text{ and } \pi^i_{r-2}(1)=\pi^i_{r-1}(1),\\
            \{\pi\in \calP([n]): \pi(1)\notin \{\pi^i_{r-2}(1),\pi^i_{r-1}(1)\}\} & \text{if } \chi^i_{r-2}=\chi^i_{r-1}=\chi^i_r=\move \text{ and } \pi^i_{r-2}(1)\neq \pi^i_{r-1}(1),\\
            \calP([n]) & \text{otherwise,}
        \end{array}
    \]\par\vspace*{-1ex}
    for all~$r\in \NN$ with~$r/3\in \NN$.
    \end{enumerate}

Focusing on rounds in which both players move,\footnote{These are, in a sense, the \textit{interesting} rounds when we are looking for an improvement: players do not necessarily meet but may be able to learn in order to better coordinate in subsequent rounds.} players fail to meet if and only if the permutations they choose are derangements. Thus, under~$\SAW$, the probability that the players have not met within~$r\in\NN$ moving rounds is~$(\hat{d}^0_{n})^r$. We will see that the probability is smaller under~$\SCFK$ when~$r\geq 3$.

\subsection{Rendezvous with Visibility}

To understand this improvement, it is helpful to view~$\SAW$ and~$\SCFK$ as operating on the complement of the derangement graph on~$\calP([n])$. Vertices in this graph are permutations in~$\calP([n])$, and two different vertices are connected by an edge if they are \emph{not} derangements. 
Rendezvous in a particular round occurs if the players choose the same vertex or two neighboring vertices.

This motivates a generalization of the rendezvous problem on a vertex-transitive \textit{visibility graph}~$G=(V,E)$, where players choose a vertex in each time step and meet if they have chosen vertices~$u$ and~$v$ such that~$u=v$ or~$\{u,v\}\in E$. We obtain the original rendezvous problem by setting~$E=\emptyset$. 
Recall that a graph $G = (V,E)$ is \emph{vertex-transitive} if for all $u,v \in V$ there is a permutation $f\colon V \to V$ with $f(u) = v$ and $(w,x) \in E$ if and only if $(f(w),f(x)) \in E$. 
Vertex-transitivity is a useful property for visibility graphs, since in a vertex-transitive graph all vertices look the same and players therefore cannot coordinate a priori on a subset of vertices. 
For a given visibility graph, a rendezvous strategy~$\sigma$, as well as the random variables~$T_\sigma$ and~$t_\sigma$, can be defined as before.\footnote{We again focus on~$t_\sigma$, conditional on the players not meeting at time~$0$, when their initial vertices are chosen independently and uniformly at random.} 
We further define~$\delta\coloneqq\frac{1}{|V|^2}|\{(u,v)\in V^2: u\neq v,\ \{u,v\}\notin E\}|$, and note that~$\delta$ is the probability that the players fail to meet in a single round when choosing vertices uniformly at random.

The round-based strategies we have introduced for the original rendezvous problem can be understood as strategies for rendezvous with visibility on the complement of the derangement graph, which is clearly vertex-transitive.
$\SAW$ chooses a permutation uniformly at random in each moving round and thus corresponds to the \textit{uniform strategy}~$\Uni$ for rendezvous with visibility, which in every step visits a vertex chosen uniformly at random. 
$\SCFK$, on the other hand, belongs to a family of strategies we will call \textit{$\ell$-clique strategies}. 
Call a partition~$V_1,\ldots, V_m$ of the set~$V$ a \textit{clique partitioning} if~$V_i$ is a clique with~$|V_i|=q$ for all~$i\in [m]$.
Related notions are studied by \citet{ErskineGS22}, whose definition of a weakly $(m,q)$-clique-partitioned graph additionally requires uniqueness of the partition; for further properties of vertex-transitive clique-partitioned graphs, see, e.g., \citet{DobsonHMV15}. 
We consider clique partitions for which there exists $\delta_{\mathrm{out}}\in\{0,\ldots,q-1\}$ such that
\begin{equation}\label{eq:uniform-clique-degrees}
    |\{v\in V_j:\{u,v\}\in E\}|=\delta_{\mathrm{out}}
    \qquad
    \text{for all distinct }i,j\in[m]\text{ and }u\in V_i.
\end{equation}
Thus, every vertex has the same number of neighbors in each of the other cliques.
In the complement of the derangement graph on $\calP([n])$, the partition $V_i=\{\pi\in\calP([n]):\pi(1)=i\}$ for $i\in[n]$ satisfies this condition with $m=n$, $q=(n-1)!$, and $\delta_{\mathrm{out}}=(n-1)!-d^1_{n-1}$.
Indeed, for any fixed $\pi\in V_i$ and $j\neq i$, exactly $d^1_{n-1}$ permutations in $V_j$ are derangements of $\pi$, by \Cref{lem:shifted-derangements}.
Given a graph with such a clique partition, and~$\ell\in\NN$ with~$\ell\leq\min\{m-1,q-1\}$, the~$\ell$-clique strategy with parameter~$\mu\in [0,1]$, denoted by~$\Vis_{\ell}(\mu)$, operates in rounds of length~$\ell$. 
For each round~$r\in\NN$, and independently of all other rounds, a player stays with probability~$\mu$ and moves with probability~$1-\mu$. The player then chooses for each step between~$(r-1)\ell+1$ and~$r\ell$ a vertex uniformly at random from a restricted subset of vertices; when staying she chooses, independently from the other steps, vertices from her \textit{home clique}, i.e., the clique containing the location she visited at step~$0$; when moving she chooses uniformly among the cliques other than her home clique that she has not yet visited in this round, and then chooses a vertex independently and uniformly from the selected clique.

\subsection{Beating the Uniform Strategy}

On an intuitive level,~$\ell$-clique strategies mirror~$\ell$-strategies like~$\AW$ in a way that exploits the structure of the visibility graph:
Players in staying rounds now stay at their home clique but move randomly within it; players in moving rounds pick both an unvisited clique and a location within it uniformly at random.
The random choice within a clique is made for simplicity, to keep the meeting probability constant for all steps of a round under the condition of visiting distinct cliques.
What enables this kind of strategy to achieve an improvement over the uniform strategy is that it shifts probability mass towards the event that both players visit the same clique when at least one of them moves, an event that guarantees rendezvous.

It is again helpful to look at the complement of the visibility graph.
The uniform strategy has a non-meeting probability, after~$\ell$ steps, equal to~$\delta^\ell$. 
On the other hand,~$\ell$-clique strategies reduce the problem to a contracted graph where the~$m$ independent sets become vertices and players choose either the same vertex or~$\ell$ distinct vertices for the~$\ell$ steps of a round.
By \eqref{eq:uniform-clique-degrees}, the non-meeting probability when visiting two distinct cliques is $1-\frac{\delta_{\mathrm{out}}}{q}=\frac{m}{m-1}\delta$.
Indeed, uniformly chosen vertices belong to distinct cliques with probability $(m-1)/m$, and vertices in the same clique always meet. 
\Cref{fig:derangement-graphs-a} shows the original and the contracted derangement graph for permutations of length~$3$.
The trade-off is now clear: compared to the uniform strategy,~$\ell$-clique strategies are played on a smaller graph---where visiting the same vertex is easier---but with lower visibility---making rendezvous harder when visiting distinct vertices. 
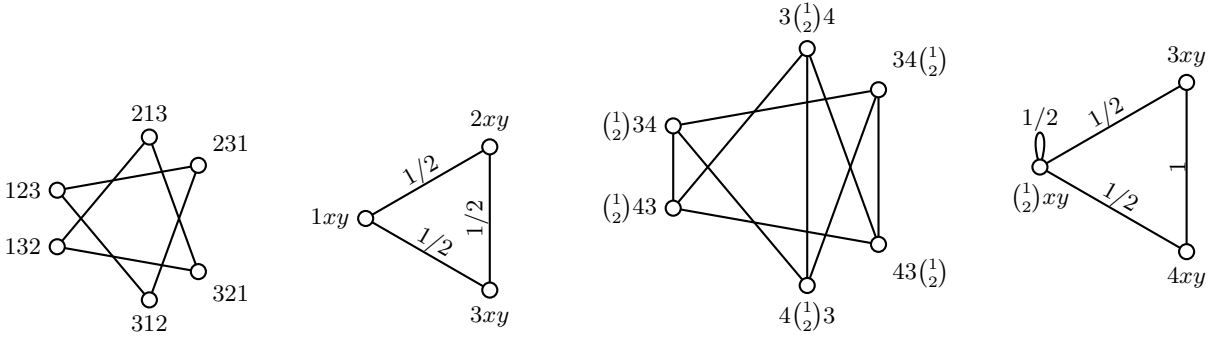
\begin{figure}[tb]
\centering
\begin{subfigure}[b]{0.44\textwidth}
\begin{tikzpicture}[every edge quotes/.style = {above=-3pt, font=\footnotesize, sloped}, every loop/.style={}]
\footnotesize
  \tikzset{v/.style={circle,draw,thick,inner sep=1pt,outer sep=0pt, minimum size=.2cm}}
  \matrix[column sep=.3cm,row sep=.7cm]{
    {\draw + (40:1.1) node[label=above right:231](1)[v]{} + (80:1.1) node[label=above:213](2)[v]{} + (160:1.1) node[label=left:123](3)[v]{} + (200:1.1) node[label=left:132](4)[v]{} + (280:1.1) node[label=below:312](5)[v]{} + (320:1.1) node[label=below right:321](6)[v]{};
	\foreach \x/\y in {1/3,2/4,3/5,4/6,5/1,6/2} {\draw[-,thick] (\x) to (\y);}} && 
    {\draw + (60:1.1) node[label=above:$2xy$](1)[v]{} + (180:1.1) node[label=left:$1xy$](2)[v]{} + (300:1.1) node[label=below:$3xy$](3)[v]{};
	\foreach \x/\y in {1/2,3/2,3/1} {\draw[-,thick] (\x) edge["$1/2$"] (\y);}}\\};
\end{tikzpicture}
\caption{Graphs representing the permutations that our simplified strategies from \Cref{sec:warm-up} play for~$n=3$ locations. The non-meeting probability when visiting two vertices taken uniformly at random in the original graph is $\delta=1/3$, and when visiting two different vertices in the new graph is $\frac{n}{n-1}\delta = \frac{1}{2}$.}
\label{fig:derangement-graphs-a}
\end{subfigure}
\hfill
\begin{subfigure}[b]{0.52\textwidth}
\begin{tikzpicture}[every edge quotes/.style = {above=-3pt, font=\footnotesize, sloped}]
\footnotesize
    \tikzset{v/.style={circle,draw,thick,inner sep=0pt,outer sep=0pt, minimum size=0.2cm},every loop/.style={}}
  \matrix[column sep=.3cm,row sep=.7cm]{
    {\draw + (40:1.6) node[label=above right:$34\binom{1}{2}$](1)[v]{} + (80:1.6) node[label=above:$3\binom{1}{2}4$](2)[v]{} + (160:1.6) node[label=left:$\binom{1}{2}34$](3)[v]{} + (200:1.6) node[label=left:$\binom{1}{2}43$](4)[v]{} + (280:1.6) node[label=below:$4\binom{1}{2}3$](5)[v]{} + (320:1.6) node[label=below right:$43\binom{1}{2}$](6)[v]{};
	\foreach \x/\y in {1/3,2/4,3/5,4/6,5/1,6/2,1/6,2/5,3/4} {\draw[-,thick] (\x) to (\y);}} && 
    {\draw + (60:1.3) node[label=above:$3xy$](1)[v]{} + (180:1.3) node[label=below:$\binom{1}{2}xy$](2)[v]{} + (300:1.3) node[label=below:$4xy$](3)[v]{};
	\foreach \x/\y in {1/2,3/2} {\draw[-,thick] (\x) edge["$1/2$"] (\y);}
    \draw[-,thick] (3) edge["$1$"] (1);
    \draw[-,thick] (2) edge[loop above,"$1/2$"] (2);
    }\\};
\end{tikzpicture}
\caption{Graphs representing the permutations that our strategies from \Cref{sec:beating} play for~$n=4$ locations. We write~$\binom{1}{2}$ for the home location of the other player; two permutations meet only if they coincide at another location ($3$ or~$4$). In particular, a self-loop appears because players are not guaranteed to meet when they both visit permutations that start with~$\binom{1}{2}$.
}
\label{fig:derangement-graphs-b}
\end{subfigure}
\caption{Examples of original and contracted derangement graphs. The original graphs contain edges between pairs of permutations that do not meet; the contracted graphs have edges between vertices that represent a set of permutations, and the label corresponds to the non-meeting probability when choosing one permutation from each endpoint uniformly at random.}
\label{fig:derangement-graphs}
\end{figure}
It turns out that lower visibility dominates when~$\ell=1$, but that the contraction pays off already when~$\ell=2$.
\begin{restatable}{proposition}{propSimpleAW}\label{prop:simple-aw}
    Let~$G=(V,E)$ admit a partition~$V_1,\ldots,V_m$ into cliques of common size~$q$, where~$m\geq3$ and~$q\geq3$, satisfying \eqref{eq:uniform-clique-degrees}.
    Then,~$\PP[t_{\Vis_1(\mu)}\leq 1] \leq \PP[t_{\Uni}\leq 1]$ for every~$\mu \in [0,1]$, and~$\PP[t_{\Vis_2(1/m)}\leq r] > \PP[t_{\Uni}\leq r]$ for every~$r\in\NN$ with~$r\geq 2$.
\end{restatable}

We prove this proposition in \Cref{app:prop:simple-aw} by showing that~$\PP[t_{\Vis_1(\mu)}>1]\geq\delta$ for every~$\mu\in[0,1]$ and that $\PP[t_{\Vis_2(1/m)}>2]<\delta^2$. 
We determine the values of these expressions as a function of the parameter~$\mu$ by calculating the probabilities that the players do not visit the same clique in the first or first two steps, conditioning on the number of moving players, and applying the law of total probability.

\Cref{prop:simple-aw} implies immediately that~$\SCFK(\theta)$ improves over~$\SAW(\theta)$ in terms of meeting probabilities after three rounds.
Indeed, when moving for three consecutive rounds,~$\SAW(\theta)$ takes all permutations uniformly at random, while~$\SCFK(\theta)$ takes the second and third permutations according to~$\Vis_2(1/m)$, with~$m=n$ cliques of size~$q=(n-1)!$ defined by fixing the first element of each permutation. 
Since these cliques have only size~$2$ when the number of locations is~$n=3$, and it is not possible to partition the associated visibility graph into~$m\geq 3$ cliques of size~$q\geq 3$, the improvement applies for~$n\geq 4$. 
\begin{corollary}
    For every~$n\in \NN$ with~$n\geq 4$, every~$\theta\in [0,1)$, and every~$r\in \NN$ with~$r\geq 3$,
    \[
        \PP[t_{\SCFK(\theta)}\leq rn] > \PP[t_{\SAW(\theta)}\leq rn].
    \]
\end{corollary}

$\SCFK$ does not improve over~$\AW$, even in terms of meeting probabilities. This is due to the fact that~$\SCFK$ visits the home location even in moving rounds and thus uses rounds of length~$n$ compared to length~$n-1$ for~$\AW$. We will resolve this issue by returning to~$\CFK$, and will also obtain an improvement of the meeting time. The analysis will become much more involved.

\section{Beating the Anderson--Weber Strategy}\label{sec:beating}

We will now show our main result, that~$\CFK$ has a strictly smaller expected meeting time than~$\AW$ for any~$n\geq 4$.
The analysis relies heavily on an intermediate result that we have made plausible in \Cref{sec:warm-up}: for~$n\geq 5$,~$\CFK$ has a higher probability of meeting than~$\AW$ after~$r$ rounds for any~$r\geq 3$. 

Before we start the technical analysis, let us briefly explain the intuition behind our improvement over~$\AW$ in light of the analysis in \Cref{sec:warm-up}.
Like its simplified version from \Cref{sec:warm-up},~$\CFK$ correlates consecutive permutations in a way that emulates an Anderson--Weber-like strategy on the complement of the derangement graph.
Specifically, for three consecutive permutations, it chooses the first two uniformly at random, and the third uniformly at random subject to the constraint that the three permutations must start with the same location or three distinct locations.
However, to actually improve over~$\AW$, it will be necessary as in~$\AW$ to use a home location in which a player stays in staying rounds and which can safely be excluded from moving rounds.
The existence of the home location considerably complicates the analysis, since now the permutations players choose have length~$n-1$ and differ in one element. 
The resulting graph still has the same vertices as the derangement graph on~$n-1$ locations, but has additional edges corresponding to players not meeting when visiting each other's home location. The partition by the first entry no longer yields a common non-meeting probability between distinct parts, and choosing permutations from the same part need not guarantee a meeting.
This is illustrated in \Cref{fig:derangement-graphs-b}.

In the following sections, we focus our analysis on the first three rounds and exploit the fact that both of the strategies we analyze restart every~$3(n-1)$ steps.
We first introduce some additional notation.
Let~$\sigma\in\{\AW(\theta),\CFK(\theta)\}$, and let~$\chi^i_{\sigma,r}\in \{\stay,\move\}$ denote the (random) decision of player~$i\in\{1,2\}$ in strategy~$\sigma$ to stay or move in the~$r$th round.
In a slight abuse of notation, we will write~$\chi^i_\sigma\in \{\stay,\move\}^2$ or~$\chi^i_\sigma\in \{\stay,\move\}^3$ for the decisions of player~$i\in \{1,2\}$ in the first two or first three rounds, and compress such~$2$- or~$3$-dimensional vectors with a repeated component to a single letter with the length as superscript; for example, we will use~$\chi^i_\sigma=\move^2$ to denote that player~$i$ moves in the first two rounds.
For~$i\in \{1,2\}$ and~$r\in [3]$, we let~$\pi^i_{\sigma,r}$ denote the (random) permutation in~$\calP([n]\setminus \{i\})$ that player~$i$ takes in round~$r$ under strategy~$\sigma$ if~$\chi^i_{\sigma,r}=\move$. 
We will write~$\meet^r_{\sigma} \coloneqq [(r-1)(n-1) < t_\sigma \leq r(n-1)]$ and~$\meet^{\leq r}_{\sigma} \coloneqq [t_\sigma \leq r(n-1)]$ for the events that under strategy~$\sigma$ the players meet for the first time \emph{during round~$r$} and \emph{up to round~$r$}, respectively.
A bar above~$\meet$ will be used to denote negation, so that in particular~$\bar{\meet}^{\leq r}_{\sigma}$ means that the players have \emph{not} met up to round~$r$.
Finally, the following events will be useful for the analysis of~$\CFK(\theta)$ when both players move throughout:
\[\begin{aligned}
    A^i_{1} & \coloneqq \big[\pi^i_{\CFK(\theta),1}(1)=\pi^i_{\CFK(\theta),2}(1)=3-i\big],\\
    A^i_{2} & \coloneqq \big[\pi^i_{\CFK(\theta),1}(1)=\pi^i_{\CFK(\theta),2}(1)\neq 3-i\big],\\
    A^i_{3} & \coloneqq \big[\pi^i_{\CFK(\theta),1}(1)\neq \pi^i_{\CFK(\theta),2}(1),\ 3-i \in \{\pi^i_{\CFK(\theta),1}(1),\pi^i_{\CFK(\theta),2}(1),\pi^i_{\CFK(\theta),3}(1)\}\big],\\
    A^i_{4} & \coloneqq \big[\pi^i_{\CFK(\theta),1}(1)\neq \pi^i_{\CFK(\theta),2}(1),\ 3-i \notin \{\pi^i_{\CFK(\theta),1}(1),\pi^i_{\CFK(\theta),2}(1),\pi^i_{\CFK(\theta),3}(1)\}\big].
\end{aligned}\]
Events~$A^i_{1}$ and~$A^i_{2}$ occur if player~$i$ starts the three tours at the same location, and if this location is or is not the home location of the other player, respectively.
Similarly,~$A^i_{3}$ and~$A^i_{4}$ occur if player~$i$ starts the three tours at different locations, and one of them or none of them is the home location of the other player, respectively.\footnote{We have compared~$\smash{\pi^i_{\CFK(\theta),1}(1)}$ and~$\smash{\pi^i_{\CFK(\theta),2}(1)}$ to define the events for brevity, but by definition of~$\CFK(\theta)$, $\smash{\pi^i_{\CFK(\theta),3}(1)}$ will be the same as these or different from both of these depending on whether they are the same or differ.} 
We finally define~$w^i_{k} \coloneqq \PP[A^i_{k}  \mid \chi^1_{\CFK(\theta)}\!=\!\chi^2_{\CFK(\theta)}\!=\!\move^3]$, for~$i\in \{1,2\}$ and~$k\in [4]$, as the probability of~$A^i_{k}$ under the condition that both players move in all of the first three rounds. These probabilities are independent of~$\theta$ and are given by
    \begin{equation}
        w^i_{1} = \frac{1}{(n-1)^2},\quad
        w^i_{2} = \frac{n-2}{(n-1)^2},\quad
        w^i_{3} = \frac{3(n-2)}{(n-1)^2},\quad
        w^i_{4} = \frac{(n-2)(n-4)}{(n-1)^2} \quad \text{for } i\in \{1,2\}.\label{eq:probs-first-locs}
    \end{equation}
Indeed, the first locations agree with probability $1/(n-1)$; conditional on agreement, their common value is uniform over the $n-1$ non-home locations.
The first locations do not agree with probability $(n-2)/(n-1)$ and, in such case, they are a uniform ordered triple of distinct non-home locations, containing the other home location with probability $3/(n-1)$.

\subsection{Meeting Probabilities}\label{subsec:meeting-prob}

We begin by looking at meeting probabilities by the end of the first three rounds, and we will see that~$\CFK$ provides an improvement over~$\AW$ for~$n\geq 5$ but not for~$n=4$.
\begin{restatable}{theorem}{thmMeetingProb}\label{thm:meeting-prob}
    For every~$n\in \NN$ with~$n\geq 5$ and every~$\theta\in [0,1)$, it holds 
    \[
        \PP[t_{\CFK(\theta)}\leq 3(n-1)] \geq \PP[t_{\AW(\theta)}\leq3(n-1)] + \frac{967(1-\theta)^6}{(n-1)^9}.
    \]
    For~$n=4$ and every~$\theta\in [0,1)$,~$\PP[t_{\CFK(\theta)}\leq 3(n-1)] = \PP[t_{\AW(\theta)}\leq 3(n-1)]$.
\end{restatable}

$\CFK(\theta)$ repeats every~$3$ rounds and coincides with~$\AW(\theta)$ in all rounds that are not multiples of~$3$. In addition, the strategy restarts independently after each block of three rounds, conditional on not having met. Thus, we have for every~$n\in \NN$ with~$n\geq 5$,~$\theta\in [0,1)$,~$k\in \NN$, and~$r\in \{0,1,2\}$ that
\[
    \PP\big[\bar{\meet}^{\leq 3k+r}_{\CFK(\theta)}\big] = \PP\big[\bar{\meet}^{\leq 3}_{\CFK(\theta)}\big]^k \PP\big[\bar{\meet}^{1}_{\AW(\theta)}\big]^r < \PP\big[\bar{\meet}^{\leq 3}_{\AW(\theta)}\big]^k \PP\big[\bar{\meet}^{1}_{\AW(\theta)}\big]^r = \PP\big[\bar{\meet}^{\leq 3k+r}_{\AW(\theta)}\big],
\]
where the inequality holds by \Cref{thm:meeting-prob}. 
We thus obtain the following corollary.
\begin{corollary}
    For every~$n,r\in \NN$ with~$n\geq 5$ and~$r\geq 3$, and~$\theta\in [0,1)$,
    \(
        \PP[t_{\CFK(\theta)}\leq r(n-1)] > \PP[t_{\AW(\theta)}\leq r(n-1)].
    \)
\end{corollary}

One of the key ingredients of the proof of \Cref{thm:meeting-prob} is the following lemma, which gives an explicit expression for the non-meeting probability under our strategy up to time~$3(n-1)$, under the condition that both players move for the first three rounds. We also give the exact value for~$n\in \{4,5,6,7\}$. 
\begin{restatable}{lemma}{lemCondProb}\label{lem:cond-prob-formula}
    For every~$n\in \NN$ with~$n\geq 4$, and~$\theta\in [0,1)$,
    \[\begin{aligned}
    & \PP \big[\bar{\meet}^{\leq 3}_{\CFK(\theta)} \;\big\vert\; \chi^1_{\CFK(\theta)}\!=\!\chi^2_{\CFK(\theta)}\!=\!\move^3 \big] = \\
    & \ \frac{1}{(n-1)^4(n-3)} \Big((n-3)(\hat{d}^0_{n-2})^3 + 6(n-2)(n-3)\hat{d}^0_{n-2} (\hat{d}^1_{n-2})^2 \\
    & \hspace{8em} + 3(n-2)(n^2-7n+13) \hat{d}^0_{n-2} (\hat{d}^2_{n-2})^2 + 2(n-2)(n-3)^2(\hat{d}^1_{n-2})^3 \\
    & \hspace{8em} + 6(n-2)(n-3)^2(\hat{d}^1_{n-2})^2 \hat{d}^2_{n-2} + 6(n-2)(n-4)(n^2-6n+10)\hat{d}^1_{n-2} (\hat{d}^2_{n-2})^2 \\
    & \hspace{8em} + (n-2)(n^4-14n^3+75n^2-185n+181) (\hat{d}^2_{n-2})^3 \Big).
    \end{aligned}\]
    In particular, this value equals~$\frac{1}{8}$ for~$n=4$,~$\frac{5}{54}$ for~$n=5$,~$\frac{439471}{5184000}$ for~$n=6$, and~$\frac{406417}{5184000}$ for~$n=7$.
\end{restatable}

Since $\hat{d}_n^k$ converges to $1/e$ as $n\to\infty$ for any fixed $k$, this non-meeting probability approaches $1/e^3$ as $n\to\infty$. This is the same as the corresponding non-meeting probability for $\AW$, so the improvement of $\CFK$ over $\AW$ vanishes as $n\to\infty$. 

We prove \Cref{lem:cond-prob-formula} in Appendix~\ref{app:lem:cond-prob-formula}. In the proof we further condition on the events~$A^1_{k}$ and~$A^2_{\ell}$ for~$k,\ell\in [4]$, whose probabilities are given by \eqref{eq:probs-first-locs}.
For each combination of these events, we then compute the non-meeting probabilities in the first three rounds by carefully analyzing the probabilities that (i) the respective first locations of the players coincide in some round, and that (ii) one of the subsequent permutations (of length~$n-2$) has a fixed point under the condition that the first locations do not coincide. 
\Cref{lem:shifted-derangements} provides the key ingredient for the latter.

We are now ready to prove \Cref{thm:meeting-prob}.
\begin{proof}[Proof of \Cref{thm:meeting-prob}]
    Fix~$n\in \NN$ with~$n\geq 4$ and~$\theta\in [0,1)$.
    For~$\sigma\in \{\AW(\theta),\CFK(\theta)\}$,
    \[
        \PP\big[\meet^{\leq3}_{\sigma}\big]=\sum_{\tilde\chi^1,\tilde\chi^2\in\{\stay,\move\}^3} \PP\big[\meet^{\leq3}_{\sigma},\ \chi^1_\sigma\!=\!\tilde\chi^1,\ \chi^2_\sigma\!=\!\tilde\chi^2\big].
    \]
    The terms of the sum on the right-hand side are equal for~$\sigma=\AW(\theta)$ and for~$\sigma=\CFK(\theta)$, with the exception of the one corresponding to~$\tilde{\chi}^1\!=\!\tilde{\chi}^2\!=\!\move^3$.
    Since~$\PP\big[\chi^1_{\sigma}\!=\!\chi^2_{\sigma}\!=\!\move^3\big]=(1-\theta)^6>0$ for~$\sigma\in \{\AW(\theta),\CFK(\theta)\}$, and since~$\PP\big[\meet^{\leq 3}_{\AW(\theta)} \;\big\vert\; \chi^1_{\AW(\theta)}\!=\!\chi^2_{\AW(\theta)}\!=\!\move^3\big] = 1-(\hat{d}^1_{n-1})^3$ by \Cref{lem:shifted-derangements}, we have that
    \[\begin{aligned}
        \PP\big[\meet^{\leq 3}_{\CFK(\theta)}\big] - \PP\big[\meet^{\leq 3}_{\AW(\theta)}\big]
        & = (1-\theta)^6\big( \PP\big[\meet^{\leq 3}_{\CFK(\theta)} \;\big\vert\; \chi^1_{\CFK(\theta)}\!=\!\chi^2_{\CFK(\theta)}\!=\!\move^3 \big] - (1-(\hat{d}^1_{n-1})^3) \big)\\
        & =  (1-\theta)^6\big((\hat{d}^1_{n-1})^3 - \PP\big[\bar{\meet}^{\leq 3}_{\CFK(\theta)} \;\big\vert\; \chi^1_{\CFK(\theta)}\!=\!\chi^2_{\CFK(\theta)}\!=\!\move^3 \big]\big).
    \end{aligned}\]
    The proof is then completed with the following lemma.
    \begin{restatable}{lemma}{lemImprovementProb}\label{lem:improvement-prob}
        For~$n=4$,~$\PP\big[\bar{\meet}^{\leq 3}_{\CFK(\theta)} \;\big\vert\; \chi^1_{\CFK(\theta)}\!=\!\chi^2_{\CFK(\theta)}\!=\!\move^3 \big] = (\hat{d}^1_{n-1})^3$.
        For~$n\in \NN$ with~$n\geq 5$, 
        \[
            (\hat{d}^1_{n-1})^3 - \PP\big[\bar{\meet}^{\leq 3}_{\CFK(\theta)} \;\big\vert\; \chi^1_{\CFK(\theta)}\!=\!\chi^2_{\CFK(\theta)}\!=\!\move^3 \big] \geq \frac{967}{(n-1)^9}.
        \]
    \end{restatable}
    We prove this lemma in Appendix~\ref{app:lem:improvement-prob}. We compute the difference exactly for~$n\in \{4,5,6,7\}$ using the expressions from \Cref{lem:cond-prob-formula}, while for~$n\geq 8$ we bound it using \Cref{lem:bounds}.
\end{proof}

\subsection{Expected Meeting Time}\label{subsec:meeting-time}

We now show that~$\CFK$ improves on~$\AW$ not just in terms of meeting probability but also in terms of expected meeting time. The following is our main result.
\begin{restatable}{theorem}{thMeetingTime}\label{thm:meeting-time}
For every \(\theta\in[0,1)\), 
\[\begin{aligned}
    \EE[t_{\AW(\theta)}]-\EE[t_{\CFK(\theta)}]
    &=
    \frac{8(1-\theta)^6}{81(8-(3\theta^2-2\theta+1)^3)} && \text{for } n=4, \text{ and}\\
    \EE[t_{\AW(\theta)}]-\EE[t_{\CFK(\theta)}]
    &\geq
    \frac{483(1-\theta)^6}{(n-1)^8}  && \text{for every } n\geq 5.
\end{aligned}\]
\end{restatable}

By replacing~$\theta$ with the value that minimizes~$\EE[t_{\AW(\theta)}]$ we obtain explicit improvements over the previous best strategy; see also \citet[Table~1]{AW90}.
For~$n=4$ the optimum value of~$\theta$ is approximately~$0.322$, which yields $\EE[t_{\AW(\theta)}]-\EE[t_{\CFK(\theta)}] \approx \frac{8 \cdot 0.678^6}{81(8-(3\cdot 0.322^2-2\cdot 0.322+1)^3)}
\approx 0.00125$.
For~$n=5$ the optimum value of~$\theta$ is approximately~$0.3012$, which yields $\EE[t_{\AW(\theta)}]-\EE[t_{\CFK(\theta)}] \geq 483\cdot 0.6988^6/4^8 \approx 0.00086$.
    
For the proof of \Cref{thm:meeting-prob} we conditioned on the respective first locations visited by a player in three consecutive rounds, and distinguished whether these locations were all the same or all different and whether they included the home location of the other player. This was enough to allow us to compute the probability of meeting within the three rounds. 
Now we will not only be concerned with the event that players meet within the three rounds, but also with the particular round and step within that round in which they meet for the first time. 
We deal with this added difficulty by exploiting that $\CFK$ and $\AW$ behave in the same way in all rounds that are not multiples of three, which allows us to focus on the expected meeting time \emph{conditioned on both players moving in the first three rounds and meeting for the first time in the third round}.
Specifically, we will show the following lemma. 
\begin{lemma}\label{lem:diff-third-round}
Let $\theta\in [0,1)$. Then
\[
    \EE\big[t_{\CFK(\theta)} \;\big\vert\; \meet^{3}_{\CFK(\theta)},\ \chi^1_{\CFK(\theta)}\!=\!\chi^2_{\CFK(\theta)}\!=\!\move^3\big] =
    \EE\big[t_{\AW(\theta)} \;\big\vert\; \meet^3_{\AW(\theta)},\ \chi^1_{\AW(\theta)} \!=\!\chi^2_{\AW(\theta)}\!=\!\move^3\big] - \frac{8}{81}
\]
if $n=4$, and
\[
     \EE\big[t_{\CFK(\theta)} \;\big\vert\; \meet^{3}_{\CFK(\theta)},\ \chi^1_{\CFK(\theta)}\!=\!\chi^2_{\CFK(\theta)}\!=\!\move^3\big] < \EE\big[t_{\AW(\theta)} \;\big\vert\; \meet^3_{\AW(\theta)},\ \chi^1_{\AW(\theta)} \!=\!\chi^2_{\AW(\theta)}\!=\!\move^3\big]
\]
if~$n\geq 5$.
\end{lemma}
Combining this lemma with \Cref{thm:meeting-prob} makes it rather intuitive that $\CFK$ beats $\AW$ in terms of meeting times. 
Indeed,~$\CFK$ and~$\AW$ have the same meeting probabilities for the first two rounds and the same meeting times conditioned on meeting within these rounds; \Cref{thm:meeting-prob} means that~$\CFK$ has a larger meeting probability in the third round if both players move in the first three rounds, and \Cref{lem:diff-third-round} implies that~$\CFK$ has a smaller expected meeting time conditioned on this event. Proving the theorem rigorously still requires some work, and we will do so at the end of the section. Most of the section will be concerned with the proof of \Cref{lem:diff-third-round}.

In order to compute the expectations in \Cref{lem:diff-third-round} we will further condition on the respective first locations the players visit in the third round, which under $\CFK$ are \emph{not} independent of not having met in previous rounds. 
We will explicitly compute these conditional expectations and the associated conditional probabilities to bound the difference between the expected meeting times of $\CFK$ and $\AW$.

For~$\theta\in [0,1)$ and~$\sigma\in \{\AW(\theta),\CFK(\theta)\}$, we define the events
\[\begin{aligned}
    B_{\sigma,1} & \coloneqq [|\{i\in \{1,2\}: \pi^i_{\sigma,3}(1)=3-i\}| = 2],\quad
    B_{\sigma,2} \coloneqq [|\{i\in \{1,2\}: \pi^i_{\sigma,3}(1)=3-i\}| = 1],\\
    B_{\sigma,3} & \coloneqq [|\{i\in \{1,2\}: \pi^i_{\sigma,3}(1)=3-i\}| = 0,\ \pi^1_{\sigma,3}(1)\neq \pi^2_{\sigma,3}(1)],\quad
    B_{\sigma,4} \coloneqq [\pi^1_{\sigma,3}(1)= \pi^2_{\sigma,3}(1)],
\end{aligned}\]
where we recall that~$\pi^i_{\sigma,r}$ denotes the permutation in~$\calP([n]\setminus\{i\})$ used by player $i\in\{1,2\}$ in round~$r\in \NN$ under strategy~$\sigma$.
Event $B_{\sigma,1}$ thus means that in the third round both players start at the home location of the other player, $B_{\sigma,2}$ that exactly one player starts at the home location of the other player, $B_{\sigma,3}$ that none of the players start at the home location of the other player but they start at distinct locations, and~$B_{\sigma,4}$ that both players start at the same location, which is of course not the home location of either player.
We will condition on~$\chi^1_{\sigma}\!=\!\chi^2_{\sigma}\!=\!\move^3$ whenever we refer to these events, so that the third permutation of a player will indeed be the permutation according to which the player visits the non-home locations.

For~$n\in \NN$ with~$n\geq 2$ and~$b\in [n-1]\cup \{0\}$, let~$\first_b(n)$ denote the expected index of the first fixed point between a permutation taken uniformly at random from~$\calP(\{1+b,\ldots,n+b\})$ and a permutation taken uniformly at random from~$\calP([n])$, independently of each other, under the condition that there is at least one fixed point between them, \ie
\[
    \first_b(n) \coloneqq \EE[\min \{i\in [n]: \pi(i)=\rho(i)\} \mid \{i\in [n]: \pi(i)=\rho(i)\} \neq \emptyset],
\]
where the expectation is taken over the choice of~$\pi\in \calP(\{1+b,\ldots,n+b\})$ and~$\rho\in \calP([n])$.
The following lemma provides an explicit expression for this expectation, and we will use it in the proofs of both \Cref{lem:diff-third-round} and \Cref{thm:meeting-time}.
\begin{restatable}{lemma}{lemExpFirstFixedPoint}\label{lem:exp-first-fixed-point}
    For every~$n\in \NN$ with~$n\geq 2$ and~$b\in [n-1]\cup \{0\}$,
    \[
        \first_b(n) = \frac{1}{1-\hat{d}^b_n} \bigg( \frac{(n+1)^2}{n+1-b} ( 1 - \hat{d}^b_{n+1}) - (n+1) \hat{d}^b_n \bigg).
    \]
\end{restatable}

The proof uses a recursive formula for the probability that there are exactly~$k\in[n-b]$ fixed points between a permutation taken uniformly at random from~$\calP(\{1+b,\ldots,n+b\})$ and a permutation taken uniformly at random from~$\calP([n])$.
This generalizes a technique used by \citet{AW90} to prove the special case where $b=0$.

\begin{proof}
Fix~$n$ and~$b$ as in the statement.
For~$m\in \NN\cup \{0\}$, $c\in [m]\cup \{0\}$, and~$\ell\in [m-c]\cup \{0\}$, we define~$p_c(m,\ell) \coloneqq \PP[|\{i\in [m]: \pi(i)=\rho(i)\}| = \ell]$, where~$\pi$ is a permutation taken uniformly at random from the set~$\calP(\{1+c,\ldots,m+c\})$ and~$\rho$ is a permutation taken uniformly at random from the set~$\calP([m])$, independently of each other.
We use the empty-permutation convention $p_0(0,0)=1$.
Note that, since $\pi$ is taken uniformly at random from~$\calP(\{1+c,\ldots,m+c\})$, this probability is the same for any fixed~$\rho\in \calP([m])$; in particular,~$p_c(m,\ell) = \PP[|\{i\in [m]: \pi(i)=i\}| = \ell]$.
We will be interested in probabilities of the form $p_b(n,k)$ for $k\in [n-b]\cup \{0\}$, in order to compute the expectation $\first_b(n)$ through the law of total expectation.

For each~$k\in [n-b]\cup\{0\}$, we have that
\begin{equation}\label{eq:prob-bnk}
\begin{aligned}
    p_b(n,k) & = \frac{1}{n!} \big| \big\{ \pi\in \calP(\{1+b,\ldots,n+b\}): |\{i\in \{1+b,\ldots,n\}: \pi(i)=i\}| = k \big\} \big|\\
    & =  \frac{1}{n!} \sum_{S\subseteq \{1+b,\ldots,n\}: |S|=k} \big| \big\{ \pi\in \calP(\{1+b,\ldots,n+b\}): \pi(i)=i\ \forall i\in S,\ \pi(i)\neq i\ \forall i\notin S  \big\} \big|\\
    & = \frac{1}{n!} \binom{n-b}{k} \big| \big\{ \pi\in \calP(\{1+b,\ldots,n+b\}):\\[-10pt]
    & \phantom{ = \frac{1}{n!} \binom{n-b}{k} \big| \big\{} \pi(i)=i\ \forall i\in \{1+b,\ldots,k+b\},\ \pi(i)\neq i\ \forall i\in \{k+1+b,\ldots,n\} \big\} \big|\\
    & = \frac{1}{n!} \binom{n-b}{k} \big| \big\{ \pi\in \calP(\{1+b,\ldots,n-k+b\}): \pi(i)\neq i\ \forall i\in \{1+b,\ldots,n-k\} \big\} \big|\\
    & =\frac{1}{n!} \binom{n-b}{k} (n-k)! \cdot p_b(n-k,0) = \frac{(n-b)!}{n!} \cdot \frac{(n-k)!}{(n-k-b)!} \cdot \frac{1}{k!} \cdot p_b(n-k,0),
\end{aligned}
\end{equation}
where the second equality follows by partitioning permutations by their fixed-point set, and the third one by observing that the terms of the sum are the same for every set~$S$, so we can take an arbitrary one and reduce to counting permutations on~$\{k+1+b,\ldots,n+b\}$ with no fixed points. The subsequent equalities follow by shifting, simplifying, and applying the definition of~$p_b(n-k,0)$.

Since every permutation in~$\calP(\{1+b,\ldots,n+b\})$ has a number of fixed points between~$0$ and~$n-b$, we know that~$\sum_{k=0}^{n-b}p_b(n,k)=1$.
Therefore, \eqref{eq:prob-bnk} yields
\begin{equation}\label{eq:prob-bn0}
    p_b(n,0) = 1-\sum_{k=1}^{n-b} \frac{(n-b)!}{n!} \cdot \frac{(n-k)!}{(n-k-b)!} \cdot \frac{1}{k!} \cdot p_b(n-k,0).
\end{equation}

We now note that, for any~$k\in[n-b]$ such that~$p_b(n,k)>0$, conditional on the fact that there are exactly~$k$ fixed points between permutations~$\pi\in \calP(\{1+b,\ldots,n+b\})$ and~$\rho\in \calP([n])$ taken uniformly at random, the expected index of the first one is simply~$\frac{n+1}{k+1}$, i.e.,\footnote{This fact was proven by \citet{AW90} for the case with~$b=0$ but does not depend on this value; we provide a proof for completeness.}
\begin{equation}\label{eq:cond-exp-first}
    \EE\big[\!\min \{i\in [n]: \pi(i)=\rho(i)\} \;\big\vert\; |\{i\in [n]: \pi(i)=\rho(i)\}| = k\big] = \frac{n+1}{k+1},
\end{equation}
where the expectation is taken with respect to~$\pi\in \calP(\{1+b,\ldots,n+b\})$ and~$\rho\in \calP([n])$.
To see this, we observe that every subset $S\subseteq [n]$ with $|S|=k$ is equally likely to be the set of fixed points, i.e., the probability
\[
    \PP\big[\{i\in [n]: \pi(i)=\rho(i)\} = S \;\big\vert\; |\{i\in [n]: \pi(i)=\rho(i)\}| = k\big]
\]
is the same for every $S\subseteq [n]$ with $|S|=k$.
Since the number of subsets of~$\{j,\ldots,n\}$ of size~$k$ is~$\binom{n+1-j}{k}$ for $j\in [n-k+1]$, the above expectation is equal to
\[
    \sum_{j=1}^{n-k+1} \frac{\binom{n+1-j}{k}}{\binom{n}{k}} = \frac{1}{\binom{n}{k}} \sum_{j=k}^{n} \binom{j}{k} = \frac{1}{\binom{n}{k}} \binom{n+1}{k+1} = \frac{n+1}{k+1},
\]
where the third equality follows from the hockey-stick identity.

We can now combine the previous equations to obtain
\begingroup
\allowdisplaybreaks
\[\begin{aligned}
    \first_b(n) & = \frac{\sum_{k\in[n-b]: p_b(n,k)>0} \EE[\min \{i\in [n]: \pi(i)=\rho(i)\} \mid |\{i\in [n]: \pi(i)=\rho(i)\}| = k] \cdot p_b(n,k)}{\PP[\{i\in [n]: \pi(i)=\rho(i)\}\neq \emptyset]}\\
    & = \frac{1}{1-p_b(n,0)}\sum_{k=1}^{n-b} \frac{n+1}{k+1}\cdot \frac{(n-b)!}{n!} \cdot \frac{(n-k)!}{(n-k-b)!} \cdot \frac{1}{k!} \cdot p_b(n-k,0)\\
    & = \frac{1}{1-p_b(n,0)} \cdot \frac{(n+1)^2}{n+1-b} \sum_{k=1}^{n-b} \frac{(n+1-b)!}{(n+1)!} \cdot \frac{(n+1-(k+1))!}{(n+1-(k+1)-b)!} \cdot \frac{1}{(k+1)!} \\[-10pt]
    & \phantom{= \frac{1}{1-p_b(n,0)} \cdot \frac{(n+1)^2}{n+1-b} \sum_{k=1}^{n-b} {}} \cdot p_b(n+1-(k+1),0)\\[-5pt]
    & = \frac{1}{1-p_b(n,0)} \cdot \frac{(n+1)^2}{n+1-b} \sum_{k=2}^{n+1-b} \frac{(n+1-b)!}{(n+1)!} \cdot \frac{(n+1-k)!}{(n+1-k-b)!} \cdot \frac{1}{k!} \cdot p_b(n+1-k,0)\\
    & = \frac{1}{1-p_b(n,0)} \bigg( \frac{(n+1)^2}{n+1-b} ( 1 - p_b(n+1,0)) - (n+1) p_b(n,0) \bigg).
\end{aligned}\]
\endgroup
where the second equality follows from \eqref{eq:prob-bnk} and \eqref{eq:cond-exp-first} and the fifth one from \eqref{eq:prob-bn0}.
The result now directly follows from \Cref{lem:shifted-derangements}, since~$p_b(n,0)=\hat{d}^b_n$ and~$p_b(n+1,0)=\hat{d}^b_{n+1}$.
\end{proof}

The proof of \Cref{lem:diff-third-round} uses two main ingredients, which we state as lemmas below. The first of these applies \Cref{lem:exp-first-fixed-point} to obtain explicit expressions for the expected meeting time under~$\AW(\theta)$ or~$\CFK(\theta)$, under the conditions that both players move for the first three rounds and meet for the first time in the third round, and depending on the first locations they visit in the third round.
\begin{restatable}{lemma}{lemCondExpectations}\label{lem:cond-expectations}
    For every $n\in \NN$ with $n\geq 4$, and $\theta\in [0,1)$, $\sigma\in \{\AW(\theta),\CFK(\theta)\}$, it holds $\EE\big[t_{\sigma} \;\big\vert\; \meet^3_{\sigma},\ \chi^1_{\sigma}\!=\!\chi^2_{\sigma}\!=\!\move^3,\ B_{\sigma,4}\big] = 2(n-1)+1$ and 
    \(
    \EE\big[t_{\sigma} \;\big\vert\; \meet^3_{\sigma},\ \chi^1_{\sigma}\!=\!\chi^2_{\sigma}\!=\!\move^3,\ B_{\sigma,j}\big] = 2(n-1)+1+\first_{j-1}(n-2)\)
    for $j\in [3]$, where $j=3$ is included only for $n\geq 5$.
\end{restatable}
\begin{proof}
    We fix~$n$,~$\theta$, and~$\sigma$ as in the statement.
    The equality for~$B_{\sigma,4}$ is trivial, since the events~$\chi^1_{\sigma}\!=\!\chi^2_{\sigma}\!=\!\move^3$, $\meet^3_{\sigma}$, and~$B_{\sigma,4}$ jointly imply that the players meet in the first step of the third round, \ie $\EE[t_{\sigma} \mid \meet^3_{\sigma},\ \chi^1_{\sigma}\!=\!\chi^2_{\sigma}\!=\!\move^3,\ B_{\sigma,4}] = 2(n-1) + 1$.
    For the other expressions, we apply \Cref{lem:exp-first-fixed-point}.
    Conditional on~$\chi^1_{\sigma}\!=\!\chi^2_{\sigma}\!=\!\move^3$, and~$B_{\sigma,j}$ for some $j\in [3]$, after the first step of round three, the remaining independent uniform tours have length $n-2$ and respectively $n-2$, $n-3$, or $n-4$ common locations for $j=1,\ 2$, or $3$, respectively. Additionally conditioning on $\meet^3_{\sigma}$, \ie meeting in that round, the expected numbers of additional steps are therefore the values $\first_0(n-2)$, $\first_1(n-2)$, and $\first_2(n-2)$. 
    Adding the first~$2(n-1)+1$ steps gives the claimed expressions.
    The restriction at $n=4$ follows because the $j=3$ residual sets are disjoint.
\end{proof}

The second main ingredient of the proof of \Cref{lem:diff-third-round}, which will allow us to use the conditional expectations in \Cref{lem:cond-expectations}, is a collection of expressions for the probabilities of the events~$B_{\sigma, k}$, for~$k\in [4]$, under the condition that both players move for the first three rounds and meet for the first time in the third round. Computing these probabilities is straightforward for $\AW$, but fairly difficult for $\CFK(\theta)$ due to correlations among rounds.

The following lemma provides explicit expressions for the probability, under~$\CFK(\theta)$ and the condition that both players move in the first three rounds, that (i) the players do not meet in the first two rounds and (ii) event~$B_{\CFK(\theta), k}$ takes place for $k\in [3]$. 
The probability of the event~$B_{\CFK(\theta), k}$ under the same condition can then be computed via Bayes' rule. 
\begin{restatable}{lemma}{lemCondProbabilities}\label{lem:cond-probabilities}
    For every~$n\in \NN$ with~$n\geq 4$ and~$\theta\in [0,1)$,
    \[\begin{aligned}
        \PP\big[B&_{\CFK(\theta),1},\ \bar{\meet}^{\leq 2}_{\CFK(\theta)} \;\big\vert\; \chi^1_{\CFK(\theta)}\!=\!\chi^2_{\CFK(\theta)}\!=\!\move^3\big] \\
        ={} & \frac{1}{(n-1)^4} \bigg( (\hat{d}^0_{n-2})^2 + 2(n-2)(\hat{d}^1_{n-2})^2 + \frac{(n-2)(n^2-7n+13)}{n-3} (\hat{d}^2_{n-2})^2 \bigg), \\[10pt]
        \PP\big[B&_{\CFK(\theta),2},\ \bar{\meet}^{\leq 2}_{\CFK(\theta)} \;\big\vert\; \chi^1_{\CFK(\theta)}\!=\!\chi^2_{\CFK(\theta)}\!=\!\move^3\big] \\
        ={} & \frac{2(n-2)}{(n-1)^4} \bigg( 2\hat{d}^0_{n-2}\hat{d}^1_{n-2} + (n-3)(\hat{d}^1_{n-2})^2 + 2(n-3) \hat{d}^1_{n-2} \hat{d}^2_{n-2} + \frac{(n-4)(n^2-6n+10)}{n-3}(\hat{d}^2_{n-2})^2 \bigg),\\[10pt]
        \PP\big[B&_{\CFK(\theta),3},\ \bar{\meet}^{\leq 2}_{\CFK(\theta)} \;\big\vert\; \chi^1_{\CFK(\theta)}\!=\!\chi^2_{\CFK(\theta)}\!=\!\move^3\big] \\
        ={} & \frac{(n-2)}{(n-1)^4} \bigg( \frac{2(n^2-7n+13)}{n-3}\hat{d}^0_{n-2}\hat{d}^2_{n-2} + 2(n-3)(\hat{d}^1_{n-2})^2 + \frac{4(n-4)(n^2-6n+10)}{n-3} \hat{d}^1_{n-2}\hat{d}^2_{n-2} \\
        & \phantom{\frac{(n-2)}{(n-1)^4} \bigg(} + \frac{n^4-14n^3+75n^2-185n+181}{n-3}(\hat{d}^2_{n-2})^2 \bigg).
    \end{aligned}\]
\end{restatable}

\begin{proof}
Let~$n$ and~$\theta$ be as in the statement.
Since the statement is only about the strategy~$\CFK(\theta)$, we omit the subscript~$\CFK(\theta)$ throughout the proof for compactness.

We obtain the three probabilities from the expression in \Cref{lem:cond-prob-formula}. 
To this end, let~$F(x_1,x_2,x_3)$ be the polynomial obtained from the right-hand side of that expression by replacing~$\hat d^0_{n-2}$, $\hat d^1_{n-2}$, and~$\hat d^2_{n-2}$ by~$x_1$, $x_2$, and~$x_3$, respectively. 
These variables represent the non-meeting probabilities in a round, conditional on the first locations of the players: $x_1$ corresponds to both players starting at the home location of the other player, $x_2$ to exactly one player doing so, and~$x_3$ to the players starting at distinct locations, neither of which is the home location of the other player.
If the first locations coincide, the players meet immediately and the non-meeting probability is zero.

More precisely, for~$r\in[3]$, let~$C_r\in[4]$ indicate which of these four cases occurs in round~$r$, using the same classification as in the definition of~$B_1,\ldots,B_4$.
In particular, $[C_3=j]=B_j$ for~$j\in[4]$.
Fixing~$x_4=0$, the proof of \Cref{lem:cond-prob-formula} gives
\(
    F(x_1,x_2,x_3) = \EE\big[ x_{C_1}x_{C_2}x_{C_3} \;\big\vert\; \chi^1=\chi^2=\move^3 \big].
\)
Indeed, conditional on the first locations of all three tours,
the remaining permutations are independent, so the non-meeting
probability is the product of the corresponding three factors.

We next observe that the joint distribution of~$(C_1,C_2,C_3)$ is unchanged by a permutation of the rounds.
To see this, fix a player and condition on her moving in all three rounds.
Each specified triple of equal first locations has probability~$1/(n-1)^2$, and each specified ordered triple of distinct first locations has probability~$1/((n-1)^2(n-3))$.
All other triples have probability zero.
Thus, the distribution of her first locations is unchanged by a permutation of the rounds.
Since the players choose their tours independently, the same holds for~$(C_1,C_2,C_3)$.

For~$j\in[3]$, differentiating the preceding expression therefore yields
\[
    \frac{\partial F}{\partial x_j}(x_1,x_2,x_3) = \sum_{r=1}^3 \EE\bigg[ \mathbf{1}_{\{C_r=j\}} \prod_{s\in[3]\setminus\{r\}}x_{C_s} \;\bigg\vert\; \chi^1=\chi^2=\move^3 \bigg] = 3\EE\big[ \mathbf{1}_{B_j}x_{C_1}x_{C_2} \;\big\vert\; \chi^1=\chi^2=\move^3 \big].
\]
Differentiation thus removes the non-meeting factor from one round and restricts that round to case~$j$.
By symmetry among the rounds, each of the three resulting terms can be identified with the term for the third round.
Substituting the actual non-meeting probabilities for the remaining two rounds gives
\[
    \PP\big[ B_j,\ \bar{\meet}^{\leq2} \;\big\vert\; \chi^1=\chi^2=\move^3 \big] = \frac13 \frac{\partial F}{\partial x_j} \big(\hat d^0_{n-2},\hat d^1_{n-2},\hat d^2_{n-2}\big).
\]

It remains to compute these derivatives.
Using the expression in \Cref{lem:cond-prob-formula}, we obtain
\[
\begin{aligned}
    \frac13\frac{\partial F}{\partial x_1} &= \frac{1}{(n-1)^4} \bigg( x_1^2+2(n-2)x_2^2 +\frac{(n-2)(n^2-7n+13)}{n-3}x_3^2
    \bigg), \\[10pt]
    \frac13\frac{\partial F}{\partial x_2} &= \frac{2(n-2)}{(n-1)^4} \bigg( 2x_1x_2+(n-3)x_2^2+2(n-3)x_2x_3 +\frac{(n-4)(n^2-6n+10)}{n-3}x_3^2 \bigg), \\[10pt]
    \end{aligned}
\]
\[
\begin{aligned}
    \frac13\frac{\partial F}{\partial x_3} &= \frac{n-2}{(n-1)^4} \bigg( \frac{2(n^2-7n+13)}{n-3}x_1x_3 +2(n-3)x_2^2 \\
    &\hspace{5em} +\frac{4(n-4)(n^2-6n+10)}{n-3}x_2x_3 +\frac{n^4-14n^3+75n^2-185n+181}{n-3}x_3^2 \bigg).
\end{aligned}
\]
Substituting~$x_j=\hat d^{j-1}_{n-2}$ for~$j\in[3]$
gives the three expressions in the statement.
\end{proof}

We now proceed with the proof of 
\Cref{lem:diff-third-round}, which consists of two main steps. 
First, we use \Cref{lem:cond-expectations} to write the conditional expectations in the statement of the lemma in terms of (i)~fractions of (shifted) derangements and (ii)~probabilities of the events~$B_{\sigma,k}$ under the conditions of moving in all three rounds and not meeting in the first two rounds.
This allows us to further write the difference between the conditional expectations under $\CFK$ and $\AW$ as a function of the differences between the conditional probabilities.
In the second step, we use \Cref{lem:cond-probabilities} to compute and analyze these differences.
Specifically, for $n\geq 5$ we show that the conditional probabilities of the events~$B_{\sigma,1}$ and~$B_{\sigma,4}$ are strictly greater under~$\CFK$ than under $\AW$, while the opposite relationship holds for the events~$B_{\sigma,2}$ and~$B_{\sigma,3}$.
Intuitively this is true because, conditional on moving in all rounds and not meeting in the first two rounds, $\CFK$ has a larger bias towards (i) meeting in the first step of the third round, which leads to the best-possible meeting time in this round, and (ii) both players visiting the home location of the other player in the first step of the third round, which leads to an earlier meeting time later in the round compared to a situation where players visit distinct locations that are not the other player's home location. 

\begin{proof}[Proof of \Cref{lem:diff-third-round}]
    Fix~$n\in \NN$ with~$n\geq 4$ and~$\theta\in [0,1)$.
    We first expand the conditional expectations in the statement by conditioning on the first location the players visit in the third round, captured by the events~$\{B_{\sigma,k}\}_{k\in [4]}$.
    Letting
    \[
        K\coloneqq \begin{cases}
            [4] & \text{if } n\geq 5,\\[-5pt]
            \{1,2,4\} & \text{if } n=4
        \end{cases}
    \]
    denote the set of indices~$k$ such that the event~$B_{\sigma,k}$ is compatible with meeting in the third round we have, for~$\sigma\in \{\AW(\theta),\CFK(\theta)\}$, that
    \begin{equation}\label{eq:cond-et-moving-meeting}
        \EE\big[t_{\sigma} \;\big\vert\; \meet^{3}_{\sigma},\ \chi^1_{\sigma}\!=\!\chi^2_{\sigma}\!=\!\move^3\big] = \sum_{k\in K}\EE\big[t_{\sigma} \;\big\vert\; \meet^3_{\sigma},\ \chi^1_{\sigma}\!=\!\chi^2_{\sigma}\!=\!\move^3,\ B_{\sigma,k}\big] \PP\big[B_{\sigma,k} \;\big\vert\; \meet^3_{\sigma},\ \chi^1_{\sigma}\!=\!\chi^2_{\sigma}\!=\!\move^3\big].
    \end{equation}
Explicit expressions for the conditional expectations on the right-hand side are given by \Cref{lem:cond-expectations}.
For the probabilities we use Bayes' rule to see that for each~$\sigma\in \{\AW(\theta),\CFK(\theta)\}$ and~$k\in K$,
\[\begin{aligned}
    \PP\big[B_{\sigma,k} \;\big\vert\; \meet^{3}_{\sigma},\ \chi^1_{\sigma}\!=\!\chi^2_{\sigma}\!=\!\move^3\big] & =  \PP\big[B_{\sigma,k} \;\big\vert\; \meet^{\leq 3}_{\sigma},\ \bar{\meet}^{\leq 2}_{\sigma},\ \chi^1_{\sigma}\!=\!\chi^2_{\sigma}\!=\!\move^3\big]\\
    & = \frac{\PP\big[\meet^{\leq 3}_\sigma \;\big\vert\; B_{\sigma,k},\ \bar{\meet}^{\leq 2}_{\sigma},\ \chi^1_{\sigma}\!=\!\chi^2_{\sigma}\!=\!\move^3\big] \PP\big[B_{\sigma,k} \;\big\vert\; \bar{\meet}^{\leq 2}_\sigma,\ \chi^1_{\sigma}\!=\!\chi^2_{\sigma}\!=\!\move^3\big]}{\PP\big[\meet^{\leq 3}_{\sigma} \;\big\vert\; \bar{\meet}^{\leq 2}_{\sigma},\ \chi^1_{\sigma}\!=\!\chi^2_{\sigma}\!=\!\move^3\big]},
\end{aligned}\]
and observe that 
\begin{equation*}
    \PP\big[\meet^{\leq 3}_\sigma \;\big\vert\; B_{\sigma,k},\ \bar{\meet}^{\leq 2}_{\sigma},\ \chi^1_{\sigma}\!=\!\chi^2_{\sigma}\!=\!\move^3\big] = \begin{cases}
        1-\hat{d}^{k-1}_{n-2} & \text{if } k\in [3]\\[-5pt]
        1 & \text{if } k=4.
    \end{cases}
\end{equation*}
To see the latter, consider a situation where both players move for three rounds and do not meet in the first two. 
If we condition on~$B_{\sigma,4}$, they meet in the first step of the third round; otherwise, they meet in the third round if the respective permutations they choose for the remaining $n-2$ steps are \emph{not} derangements.
If we condition on~$B_{\sigma,1}$, both of these permutations are chosen uniformly at random from~$\calP(\{3,\ldots,n\})$, so by \Cref{lem:shifted-derangements} they are \emph{not} derangements with probability~$1-\hat{d}^0_{n-2}$.
If we condition on~$B_{\sigma,2}$, and assuming w.l.o.g.\@ that~$\pi^1_{\sigma,3}(1)=2$ and~$\pi^2_{\sigma,3}(1)=j\neq 1$, the permutation of player~$1$ is chosen uniformly at random from~$\calP(\{3,\ldots,n\})$ and that of player~$2$ from~$\calP([n]\setminus \{2,j\})$, so by \Cref{lem:shifted-derangements} they are \emph{not} derangements with probability~$1-\hat{d}^1_{n-2}$.
If we condition on~$B_{\sigma,3}$, and assuming w.l.o.g.\@ that~$\pi^1_{\sigma,3}(1)=j\neq 2$ and~$\pi^2_{\sigma,3}(1)=k\notin \{1,j\}$, the permutation of player~$1$ is chosen uniformly at random from~$\calP([n]\setminus \{1,j\})$ and that of player~$2$ from~$\calP([n]\setminus \{2,k\})$, so by \Cref{lem:shifted-derangements} they are \emph{not} derangements with probability~$1-\hat{d}^2_{n-2}$. 
Replacing expressions in \eqref{eq:cond-et-moving-meeting} accordingly, we conclude that, for $\sigma\in \{\AW(\theta),\CFK(\theta)\}$,
\begin{equation}\label{eq:cond-et-moving-meeting2}
\begin{aligned}
\big(\EE\big[t_{\sigma} \;\big\vert\; \meet^{3}_{\sigma},\ \chi^1_{\sigma}\!=\!\chi^2_{\sigma}\!&=\!\move^3\big] - 2(n-1) -1\big) \PP\big[\meet^{\leq 3}_{\sigma} \;\big\vert\; \bar{\meet}^{\leq 2}_{\sigma},\ \chi^1_{\sigma}\!=\!\chi^2_{\sigma}\!=\!\move^3\big] =\\
&
(n-1) (1-\hat{d}^0_{n-1}-\hat{d}^0_{n-2})\PP\big[B_{\sigma,1} \;\big\vert\; \bar{\meet}^{\leq 2}_\sigma,\ \chi^1_{\sigma}\!=\!\chi^2_{\sigma}\!=\!\move^3\big]\\
& + \bigg(\frac{(n-1)^2}{n-2} (1-\hat{d}^1_{n-1})-(n-1)\hat{d}^1_{n-2}\bigg)\PP\big[B_{\sigma,2} \;\big\vert\; \bar{\meet}^{\leq 2}_\sigma,\ \chi^1_{\sigma}\!=\!\chi^2_{\sigma}\!=\!\move^3\big]\\
& + \mathbf{1}_{n\geq 5} \bigg(\frac{(n-1)^2}{n-3} (1-\hat{d}^2_{n-1})-(n-1)\hat{d}^2_{n-2}\bigg) \PP\big[B_{\sigma,3} \;\big\vert\; \bar{\meet}^{\leq 2}_\sigma,\ \chi^1_{\sigma}\!=\!\chi^2_{\sigma}\!=\!\move^3\big].
\end{aligned}
\end{equation}

For $k\in [3]$, let
\[\begin{aligned}
    \Delta_k &\coloneqq \PP\big[B_{\AW(\theta),k} \;\big\vert\; \bar{\meet}^{\leq 2}_{\AW(\theta)}, \chi^1_{\AW(\theta)}\!=\!\chi^2_{\AW(\theta)}\!=\!\move^3\big]\! -\! \PP\big[B_{\CFK(\theta),k} \;\big\vert\; \bar{\meet}^{\leq 2}_\CFK(\theta), \chi^1_{\CFK(\theta)}\!=\!\chi^2_{\CFK(\theta)}\!=\!\move^3\big].
\end{aligned}\]

The following lemma, establishing bounds on the differences $\{\Delta_k\}_{k\in[3]}$, is shown in Appendix~\ref{app:claim:diff-probabilities} using \Cref{lem:cond-probabilities}. 
\begin{restatable}{lemma}{claimDiffProbabilities}\label{claim:diff-probabilities}
    If~$n= 4$, then~$\Delta_2 = -2\Delta_1 = 8/81$. If~$n\geq 5$, then~$\Delta_2 = -2\Delta_1 > 0$ and~$\Delta_3>0$.
\end{restatable}

We now consider the cases where $n=4$ and $n\geq 5$ in turn. If~$n=4$, then by \Cref{thm:meeting-prob},
\begin{equation}\label{eq:cond-prob-n4}
    \PP\big[\meet^{\leq 3}_{\CFK(\theta)} \;\big\vert\; \bar{\meet}^{\leq 2}_{\CFK(\theta)},\ \chi^1_{\CFK(\theta)}\!=\!\chi^2_{\CFK(\theta)}\!=\!\move^3\big] = \PP\big[\meet^{\leq 3}_{\AW(\theta)} \;\big\vert\; \bar{\meet}^{\leq 2}_{\AW(\theta)},\ \chi^1_{\AW(\theta)}\!=\!\chi^2_{\AW(\theta)}\!=\!\move^3\big]
    = \frac{1}{2},
\end{equation}
because the last probability equals $1-\hat{d}^1_3$. Thus
\begingroup
\allowdisplaybreaks
\[\begin{aligned}
    \EE\big[t_{\AW(\theta)} \;\big\vert\; & \meet^3_{\AW(\theta)},\ \chi^1_{\AW(\theta)}\!=\!\chi^2_{\AW(\theta)}\!=\!\move^3\big] - \EE\big[t_{\CFK(\theta)} \;\big\vert\; \meet^{3}_{\CFK(\theta)},\ \chi^1_{\CFK(\theta)}\!=\!\chi^2_{\CFK(\theta)}\!=\!\move^3\big]\\
    & = 2\Big(\big(\EE\big[t_{\AW(\theta)} \;\big\vert\; \meet^{3}_{\AW(\theta)},\ \chi^1_{\AW(\theta)}\!=\!\chi^2_{\AW(\theta)}\!=\!\move^3\big] - 2(n-1) -1\big) \\[-2pt]
    & \hspace{7em} \cdot \PP\big[\meet^{\leq 3}_{\AW(\theta)} \;\big\vert\; \bar{\meet}^{\leq 2}_{\AW(\theta)},\ \chi^1_{\AW(\theta)}\!=\!\chi^2_{\AW(\theta)}\!=\!\move^3\big]\\
    & \hspace{2em} - \big(\EE\big[t_{\CFK(\theta)} \;\big\vert\; \meet^{3}_{\CFK(\theta)},\ \chi^1_{\CFK(\theta)}\!=\!\chi^2_{\CFK(\theta)}\!=\!\move^3\big] - 2(n-1) -1\big) \\[-2pt]
    & \hspace{7em} \cdot \PP\big[\meet^{\leq 3}_{\CFK(\theta)} \;\big\vert\; \bar{\meet}^{\leq 2}_{\CFK(\theta)},\ \chi^1_{\CFK(\theta)}\!=\!\chi^2_{\CFK(\theta)}\!=\!\move^3\big]\Big)\\
     & = 2\bigg( (n-1) (1-\hat{d}^0_{n-1}-\hat{d}^0_{n-2})\Delta_1 + \bigg(\frac{(n-1)^2}{n-2} (1-\hat{d}^1_{n-1})-(n-1)\hat{d}^1_{n-2}\bigg)\Delta_2\bigg) \\
    & = \frac{8(n-1)}{81}\bigg( \frac{2(n-1)}{n-2} (1-\hat{d}^1_{n-1})-2\hat{d}^1_{n-2} - (1-\hat{d}^0_{n-1}-\hat{d}^0_{n-2})\bigg)
    = \frac{8}{81},
\end{aligned}\]
\endgroup
where the first equality follows from \eqref{eq:cond-prob-n4}, the second from \eqref{eq:cond-et-moving-meeting2}, the third from \Cref{claim:diff-probabilities}, and the last from simple calculations. This completes the proof for~$n=4$.

If~$n\geq 5$, we observe that
\begin{equation}\label{ineq:et-diff-ngeq5}
\begin{aligned}
    & \big(\EE\big[t_{\AW(\theta)} \;\big\vert\; \meet^{3}_{\AW(\theta)},\ \chi^1_{\AW(\theta)}\!=\!\chi^2_{\AW(\theta)}\!=\!\move^3\big] - 2(n-1)-1\big) \\
    & \hspace{7em} \cdot \PP\big[\meet^{\leq 3}_{\AW(\theta)} \;\big\vert\; \bar{\meet}^{\leq 2}_{\AW(\theta)},\ \chi^1_{\AW(\theta)}\!=\!\chi^2_{\AW(\theta)}\!=\!\move^3\big]\\
    & - \big(\EE\big[t_{\CFK(\theta)} \;\big\vert\; \meet^{3}_{\CFK(\theta)},\ \chi^1_{\CFK(\theta)}\!=\!\chi^2_{\CFK(\theta)}\!=\!\move^3\big] - 2(n-1) -1\big)\\
    & \hspace{7em} \cdot \PP\big[\meet^{\leq 3}_{\CFK(\theta)} \;\big\vert\; \bar{\meet}^{\leq 2}_{\CFK(\theta)},\ \chi^1_{\CFK(\theta)}\!=\!\chi^2_{\CFK(\theta)}\!=\!\move^3\big]\\
    &  = (n-1) (1-\hat{d}^0_{n-1}-\hat{d}^0_{n-2}) \Delta_1 + \bigg(\frac{(n-1)^2}{n-2} (1-\hat{d}^1_{n-1})-(n-1)\hat{d}^1_{n-2}\bigg)\Delta_2 \\
    & \hspace{2em} + \bigg(\frac{(n-1)^2}{n-3} (1-\hat{d}^2_{n-1})-(n-1)\hat{d}^2_{n-2}\bigg)\Delta_3\\
    & > \frac{n-1}{2}\bigg( \frac{2(n-1)}{n-2} (1-\hat{d}^1_{n-1})-2\hat{d}^1_{n-2} - (1-\hat{d}^0_{n-1}-\hat{d}^0_{n-2})\bigg) \Delta_2.
\end{aligned}
\end{equation}
Indeed, the equality follows from \eqref{eq:cond-et-moving-meeting2} and the inequality from \Cref{claim:diff-probabilities}.
Since~$\Delta_2>0$ from \Cref{claim:diff-probabilities}, the expression on the right-hand side of \eqref{ineq:et-diff-ngeq5} is non-negative if the expression in parentheses is non-negative.
This is easy to check when $n=5$, where the expression in parentheses evaluates to
$\frac{11}{72} >0$.
If~$n\geq 6$, by \Cref{lem:bounds} the expression in parentheses is at least
\[\begin{aligned}
    & \frac{2(n-1)}{n-2} \bigg(1- \bigg(\frac{1}{e}+\frac{1}{(n-1)e} + \frac{1}{3(n-1)!}\bigg)\bigg)-2\bigg(\frac{1}{e}+\frac{1}{(n-2)e} + \frac{1}{3(n-2)!}\bigg)- 1 \\
    & \hspace{1em} +\bigg(\frac{1}{e}- \frac{1}{2(n-1)!}\bigg) +\bigg(\frac{1}{e}- \frac{1}{2(n-2)!}\bigg) \\
    & \hspace{1em} ={} 1-\frac{2}{e} \bigg(1+\frac{3-e}{n-2}\bigg) - \frac{7}{6(n-2)!} - \frac{2}{3(n-2)(n-2)!} - \frac{1}{2(n-1)!}.
\end{aligned}\]
The last expression is approximately $0.1527$ for $n=6$ and trivially increasing in~$n$.

Since we have shown that the expression on the left-hand side of \eqref{ineq:et-diff-ngeq5} is strictly positive, and since $\PP\big[\meet^{\leq 3}_{\AW(\theta)} \;\big\vert\; \bar{\meet}^{\leq 2}_{\AW(\theta)},\ \chi^1_{\AW(\theta)}\!=\!\chi^2_{\AW(\theta)}\!=\!\move^3\big] < \PP\big[\meet^{\leq 3}_{\CFK(\theta)} \;\big\vert\; \bar{\meet}^{\leq 2}_{\CFK(\theta)},\ \chi^1_{\CFK(\theta)}\!=\!\chi^2_{\CFK(\theta)}\!=\!\move^3\big]$ by \Cref{thm:meeting-prob}, we conclude that 
\[
    \EE\big[t_{\AW(\theta)} \;\big\vert\; \meet^3_{\AW(\theta)},\  \chi^1_{\AW(\theta)}\!=\!\chi^2_{\AW(\theta)}\!=\!\move^3\big] - \EE\big[t_{\CFK(\theta)} \;\big\vert\; \meet^{3}_{\CFK(\theta)},\ \chi^1_{\CFK(\theta)}\!=\!\chi^2_{\CFK(\theta)}\!=\!\move^3\big]
\]
is lower-bounded by the expression on the left-hand side of \eqref{ineq:et-diff-ngeq5}, and is thus strictly positive.
\end{proof}

We now proceed with the proof of \Cref{thm:meeting-time}.
Our goal will be to move from expected meeting times under strategy~$\sigma \in \{\AW(\theta),\CFK(\theta)\}$, conditioned on the events~$\meet^3_{\sigma}$ and~$\chi^1_{\sigma}\!=\!\chi^2_{\sigma}\!=\!\move^3$, to unconditional expected meeting times.
To this end, we first exploit the fact that both $\AW$ and $\CFK$ restart every~$3(n-1)$ steps to write the unconditional expectations in terms of conditional expectations of the form~$\EE\big[t_{\sigma} \;\big\vert\; \meet^{\leq 2}_{\sigma}\big]$, which are the same for both strategies, and~$\EE\big[t_{\sigma} \;\big\vert\; \meet^{3}_{\sigma}\big]$, whose difference we bound in a second step.
To express these conditional expectations in terms of the expectations that also condition on events of the type~$\chi^1_{\sigma}\!=\!\chi^2_{\sigma}\!=\!\move^3$, we proceed differently for~$n=4$ and for~$n\geq 5$.
For~$n=4$ the meeting probabilities within each round are the same under both strategies due to \Cref{thm:meeting-prob}, which simplifies the analysis.
For~$n\geq 5$ we show that, under the condition of meeting in the third round, (i)~the expected meeting time is lower when both players move in the first three rounds and (ii)~this event has a higher probability under~$\CFK(\theta)$ than under~$\AW(\theta)$ due to \Cref{thm:meeting-prob}. 
\begin{proof}[Proof of \Cref{thm:meeting-time}]
    We fix~$\theta\in(0,1)$ throughout the proof and make two initial observations.
    First, for~$\sigma\in \{\AW(\theta),\CFK(\theta)\}$,
    \[
        \EE[t_{\sigma}] = \EE\big[t_{\sigma}\;\big\vert\; \meet^{\leq 2}_\sigma\big] \PP\big[\meet^{\leq 2}_\sigma\big] + \EE\big[t_{\sigma}\;\big\vert\; \meet^{3}_\sigma\big] \PP\big[\meet^{3}_\sigma\big] + (3(n-1)+\EE[t_\sigma])\big(1-\PP\big[\meet^{\leq 3}_\sigma\big]\big),
    \]
    because both strategies restart after~$3(n-1)$ steps and ~$\meet^{\leq 2}_\sigma\,\dot\cup\, \meet^{3}_\sigma=\meet^{\leq 3}_\sigma$. Therefore,
    \begin{equation}\label{eq:et-markovian}
    \begin{aligned}
        \EE[t_{\sigma}] &= \frac{\EE\big[t_{\sigma}\;\big\vert\; \meet^{\leq 2}_\sigma\big] \PP\big[\meet^{\leq 2}_\sigma\big] + \EE\big[t_{\sigma}\;\big\vert\; \meet^{3}_\sigma\big] \PP\big[\meet^{3}_\sigma\big]}{\PP\big[\meet^{\leq 3}_\sigma\big]} + 3(n-1)\bigg(\frac{1}{\PP\big[\meet^{\leq 3}_\sigma\big]}-1\bigg)\\
        & = \EE\big[t_{\sigma}\;\big\vert\; \meet^{\leq 2}_\sigma\big] + \big(\EE\big[t_{\sigma}\;\big\vert\; \meet^{3}_\sigma\big] - \EE\big[t_{\sigma}\;\big\vert\; \meet^{\leq 2}_\sigma\big] \big)\PP\big[\meet^{3}_\sigma \;\big\vert\; \meet^{\leq 3}_\sigma \big] + 3(n-1)\bigg(\frac{1}{\PP\big[\meet^{\leq 3}_\sigma\big]}-1\bigg). 
    \end{aligned}
    \end{equation}
    Second,
    \[\begin{aligned}
        \EE\big[t&_{\AW(\theta)} \;\big\vert\; \meet^{3}_{\AW(\theta)}\big] - \EE\big[t_{\CFK(\theta)} \;\big\vert\; \meet^{3}_{\CFK(\theta)}\big]\\
        ={} & \EE\big[t_{\AW(\theta)} \;\big\vert\; \meet^{3}_{\AW(\theta)},\ \neg \big(\chi^1_{\AW(\theta)}\!=\!\chi^2_{\AW(\theta)}\!=\!\move^3\big) \big]\PP\big[ \neg\big(\chi^1_{\AW(\theta)}\!=\!\chi^2_{\AW(\theta)}\!=\!\move^3\big) \;\big\vert\; \meet^{3}_{\AW(\theta)} \big] \\
        & - \EE\big[t_{\CFK(\theta)} \;\big\vert\; \meet^{3}_{\CFK(\theta)},\ \neg\big(\chi^1_{\CFK(\theta)}\!=\!\chi^2_{\CFK(\theta)}\!=\!\move^3\big) \big] \PP\big[ \neg\big(\chi^1_{\CFK(\theta)}\!=\!\chi^2_{\CFK(\theta)}\!=\!\move^3\big) \;\big\vert\; \meet^{3}_{\CFK(\theta)} \big]\\
        & + \EE\big[t_{\AW(\theta)} \;\big\vert\; \meet^{3}_{\AW(\theta)},\ \chi^1_{\AW(\theta)}\!=\!\chi^2_{\AW(\theta)}\!=\!\move^3 \big] \PP\big[ \chi^1_{\AW(\theta)}\!=\!\chi^2_{\AW(\theta)}\!=\!\move^3 \;\big\vert\; \meet^{3}_{\AW(\theta)} \big]\\
        & - \EE\big[t_{\CFK(\theta)} \;\big\vert\; \meet^{3}_{\CFK(\theta)},\ \chi^1_{\CFK(\theta)}\!=\!\chi^2_{\CFK(\theta)}\!=\!\move^3 \big] \PP\big[ \chi^1_{\CFK(\theta)}\!=\!\chi^2_{\CFK(\theta)}\!=\!\move^3 \;\big\vert\; \meet^{3}_{\CFK(\theta)} \big],\\
        \end{aligned}\]
    and thus
    \begin{equation}\label{eq:diff-et-cond-third-round}
    \begin{aligned}
        \EE\big[t&_{\AW(\theta)} \;\big\vert\; \meet^{3}_{\AW(\theta)}\big] - \EE\big[t_{\CFK(\theta)} \;\big\vert\; \meet^{3}_{\CFK(\theta)}\big]\\
        ={} & \EE\big[t_{\AW(\theta)} \;\big\vert\; \meet^{3}_{\AW(\theta)},\ \neg \big(\chi^1_{\AW(\theta)}\!=\!\chi^2_{\AW(\theta)}\!=\!\move^3\big) \big]\\
        & \qquad \cdot \big( \PP\big[ \chi^1_{\CFK(\theta)}\!=\!\chi^2_{\CFK(\theta)}\!=\!\move^3 \;\big\vert\; \meet^{3}_{\CFK(\theta)} \big]- \PP\big[ \chi^1_{\AW(\theta)}\!=\!\chi^2_{\AW(\theta)}\!=\!\move^3 \;\big\vert\; \meet^{3}_{\AW(\theta)} \big] \big)\\
        & + \EE\big[t_{\AW(\theta)} \;\big\vert\; \meet^{3}_{\AW(\theta)},\ \chi^1_{\AW(\theta)}\!=\!\chi^2_{\AW(\theta)}\!=\!\move^3 \big] \PP\big[ \chi^1_{\AW(\theta)}\!=\!\chi^2_{\AW(\theta)}\!=\!\move^3 \;\big\vert\; \meet^{3}_{\AW(\theta)} \big]\\
        & - \EE\big[t_{\CFK(\theta)} \;\big\vert\; \meet^{3}_{\CFK(\theta)},\ \chi^1_{\CFK(\theta)}\!=\!\chi^2_{\CFK(\theta)}\!=\!\move^3 \big] \PP\big[ \chi^1_{\CFK(\theta)}\!=\!\chi^2_{\CFK(\theta)}\!=\!\move^3 \;\big\vert\; \meet^{3}_{\CFK(\theta)} \big],
    \end{aligned}
    \end{equation}
    because conditioned on meeting in the third round, the expected meeting times under both strategies are the same unless players move for all three rounds.\footnote{This might not seem entirely obvious when one player moves in the first three rounds and the other one moves in the first two and stays in the last one. In this case, it holds because conditioning on both players moving and not meeting in the first two rounds does not affect the probability of visiting the other's home location at the beginning of the third round. This can be easily seen by applying Bayes' rule and observing that the event of a player visiting the other's home location in the first step of a round does not affect the meeting probability within that round.}

    We now consider the cases where $n=4$ and $n\geq 5$ in turn. If $n=4$, we claim that
    \begin{equation}\label{eq:diff-et-n4}
    \begin{aligned}
        \EE&[t_{\AW(\theta)}] -\EE[t_{\CFK(\theta)}]\\
        & = \big( \EE\big[t_{\AW(\theta)}\;\big\vert\; \meet^{3}_{\AW(\theta)}\big] - \EE\big[t_{\CFK(\theta)}\;\big\vert\; \meet^{3}_{\CFK(\theta)}\big] \big) \PP\big[\meet^3_{\AW(\theta)} \;\big\vert\; \meet^{\leq 3}_{\AW(\theta)}\big]\\
        & = \big( \EE\big[t_{\AW(\theta)} \;\big\vert\; \meet^{3}_{\AW(\theta)}, \chi^1_{\AW(\theta)}\!=\!\chi^2_{\AW(\theta)}\!=\!\move^3 \big] - \EE\big[t_{\CFK(\theta)} \;\big\vert\; \meet^{3}_{\CFK(\theta)}, \chi^1_{\CFK(\theta)}\!=\!\chi^2_{\CFK(\theta)}\!=\!\move^3 \big] \big)\\
        & \qquad \cdot \PP\big[ \chi^1_{\AW(\theta)}\!=\!\chi^2_{\AW(\theta)}\!=\!\move^3 \;\big\vert\; \meet^{3}_{\AW(\theta)} \big] \PP\big[\meet^3_{\AW(\theta)} \;\big\vert\; \meet^{\leq 3}_{\AW(\theta)}\big]\\
        & = \frac{8}{81}\cdot \frac{\PP\big[ \meet^{3}_{\AW(\theta)} \;\big\vert\; \chi^1_{\AW(\theta)}\!=\!\chi^2_{\AW(\theta)}\!=\!\move^3 \big] \PP\big[\chi^1_{\AW(\theta)}\!=\!\chi^2_{\AW(\theta)}\!=\!\move^3 \big]}{\PP\big[\meet^{\leq 3}_{\AW(\theta)}\big]}\\
        & = \frac{8}{81}\cdot \frac{(1-\theta)^6 (\hat{d}^1_3)^2(1-\hat{d}^1_3)}{\PP\big[\meet^{\leq 3}_{\AW(\theta)}\big]} = \frac{(1-\theta)^6}{81\cdot \PP\big[\meet^{\leq 3}_{\AW(\theta)}\big]}.
    \end{aligned}
    \end{equation}
    Indeed, the first equality follows from \eqref{eq:et-markovian} and the fact that, by definition of the strategies and \Cref{thm:meeting-prob},
    \[
        \EE\big[t_{\AW(\theta)}\;\big\vert\; \meet^{\leq 2}_{\AW(\theta)}\big] = \EE\big[t_{\CFK(\theta)}\;\big\vert\; \meet^{\leq 2}_{\CFK(\theta)}\big],\
        \PP\big[\meet^{3}_{\AW(\theta)}\big] = \PP\big[\meet^{3}_{\CFK(\theta)}\big], \text{ and }
        \PP\big[\meet^{\leq 3}_{\AW(\theta)}\big] = \PP\big[\meet^{\leq 3}_{\CFK(\theta)}\big].
    \]
    The second equality follows from \eqref{eq:diff-et-cond-third-round} and the fact that, by definition of the strategies and \Cref{thm:meeting-prob},
    \[\begin{aligned}
        \PP\big[ \chi^1_{\AW(\theta)}\!=\chi^2_{\AW(\theta)}\!=\!\move^3 \;\big\vert\; \meet^{3}_{\AW(\theta)} \big]
        & = \frac{\PP\big[ \meet^{3}_{\AW(\theta)} \;\big\vert\; \chi^1_{\AW(\theta)}\!=\!\chi^2_{\AW(\theta)}\!=\!\move^3 \big] \PP\big[\chi^1_{\AW(\theta)}\!=\!\chi^2_{\AW(\theta)}\!=\!\move^3\big]}{\PP\big[ \meet^{3}_{\AW(\theta)}\big]} \\
        & = \frac{\PP\big[ \meet^{3}_{\CFK(\theta)} \;\big\vert\; \chi^1_{\CFK(\theta)}\!=\!\chi^2_{\CFK(\theta)}\!=\!\move^3 \big] \PP\big[\chi^1_{\CFK(\theta)}\!=\!\chi^2_{\CFK(\theta)}\!=\!\move^3\big]}{\PP\big[ \meet^{3}_{\CFK(\theta)}\big]} \\
        & = \PP\big[ \chi^1_{\CFK(\theta)}\!=\!\chi^2_{\CFK(\theta)}\!=\!\move^3 \;\big\vert\; \meet^{3}_{\CFK(\theta)} \big].
    \end{aligned}\]
    The third equality holds by \Cref{lem:diff-third-round} and Bayes' rule, the fourth by definition of~$\AW(\theta)$ and \Cref{lem:shifted-derangements}, and the last one because $\hat{d}^1_3=1/2$.
    Since
    \[
        \PP\big[ \meet^{\leq 3}_{\AW(\theta)} \big] = 1 - \PP\big[ \bar{\meet}^{1}_{\AW(\theta)} \big]^3 = 1 - \big(\theta^2 + (1-\theta)^2\hat{d}^1_3\big)^3 = 1 - \frac{1}{8}(3\theta^2-2\theta+1)^3,
    \]
    we conclude by \eqref{eq:diff-et-n4} that
    \(
        \EE[t_{\AW(\theta)} ] -\EE[t_{\CFK(\theta)}] = \frac{8(1-\theta)^6}{81(8 - (3\theta^2-2\theta+1)^3)}.
    \)
    This finishes the case $n=4$.
    
    For~$n\geq 5$, we claim that
    \begin{equation}\label{ineq:et-cond-third-round}
        \EE\big[t_{\AW(\theta)} \;\big\vert\; \meet^{3}_{\AW(\theta)}\big] \geq \EE\big[t_{\CFK(\theta)} \;\big\vert\; \meet^{3}_{\CFK(\theta)}\big].
    \end{equation}
    Indeed, by \eqref{eq:diff-et-cond-third-round} and \Cref{lem:diff-third-round},
    \[\begin{aligned}
        \EE\big[t&_{\AW(\theta)} \;\big\vert\; \meet^{3}_{\AW(\theta)}\big] - \EE\big[t_{\CFK(\theta)} \;\big\vert\; \meet^{3}_{\CFK(\theta)}\big]\\
        \geq{} & \big( \PP\big[ \chi^1_{\CFK(\theta)}\!=\!\chi^2_{\CFK(\theta)}\!=\!\move^3 \;\big\vert\; \meet^{3}_{\CFK(\theta)} \big]- \PP\big[ \chi^1_{\AW(\theta)}\!=\!\chi^2_{\AW(\theta)}\!=\!\move^3 \;\big\vert\; \meet^{3}_{\AW(\theta)} \big] \big)\\
        & \cdot \big( \EE\big[t_{\AW(\theta)} \;\big\vert\; \meet^{3}_{\AW(\theta)},\ \neg \big(\chi^1_{\AW(\theta)}\!=\!\chi^2_{\AW(\theta)}\!=\!\move^3\big) \big] - \EE\big[t_{\AW(\theta)} \;\big\vert\; \meet^{3}_{\AW(\theta)},\ \chi^1_{\AW(\theta)}\!=\!\chi^2_{\AW(\theta)}\!=\!\move^3 \big]\big),
    \end{aligned}\]
    and we will see that both terms in parentheses on the right-hand side are non-negative.
    For the first term this holds because
    \[\begin{aligned}
        \PP\big[ \chi&^1_{\CFK(\theta)}\!=\!\chi^2_{\CFK(\theta)}\!=\!\move^3 \;\big\vert\; \meet^{3}_{\CFK(\theta)} \big]- \PP\big[ \chi^1_{\AW(\theta)}\!=\!\chi^2_{\AW(\theta)}\!=\!\move^3 \;\big\vert\; \meet^{3}_{\AW(\theta)} \big]\\
        ={} & \PP\big[ \neg\big(\chi^1_{\AW(\theta)}\!=\!\chi^2_{\AW(\theta)}\!=\!\move^3\big) \;\big\vert\; \meet^{3}_{\AW(\theta)} \big] - \PP\big[ \neg(\chi^1_{\CFK(\theta)}\!=\!\chi^2_{\CFK(\theta)}\!=\!\move^3\big) \;\big\vert\; \meet^{3}_{\CFK(\theta)} \big]\\
        ={} & \PP\big[\meet^{3}_{\AW(\theta)} \;\big\vert\; \neg\big(\chi^1_{\AW(\theta)}\!=\!\chi^2_{\AW(\theta)}\!=\!\move^3\big) \big]\PP\big[ \neg\big(\chi^1_{\AW(\theta)}\!=\!\chi^2_{\AW(\theta)}\!=\!\move^3\big) \big] \bigg(\frac{1}{\PP\big[\meet^{3}_{\AW(\theta)}\big]} - \frac{1}{\PP\big[\meet^{3}_{\CFK(\theta)}\big]} \bigg)\\
        \geq {} & 0,
    \end{aligned}\]
    where the second equality follows from Bayes' rule and the fact that the probabilities $\PP\big[\meet^{3}_{\sigma} \;\big\vert\; \neg\big(\chi^1_{\sigma}\!=\!\chi^2_{\sigma}\!=\!\move^3\big) \big]$ and $\PP\big[ \neg\big(\chi^1_{\sigma}\!=\!\chi^2_{\sigma}\!=\!\move^3\big) \big]$ are identical for $\sigma = \AW(\theta)$ and $\sigma=\CFK(\theta)$, and the inequality from \Cref{thm:meeting-prob} and the fact that~$\PP\big[\meet^{\leq 2}_{\AW(\theta)}\big]=\PP\big[\meet^{\leq 2}_{\CFK(\theta)}\big]$.
    For the second term we note that for an event~$D\in \big\{ \big[ \chi^1_{\AW(\theta)}\!=\!\chi^2_{\AW(\theta)}\!=\!\move^3 \big], \big[ \neg\big(\chi^1_{\AW(\theta)}\!=\!\chi^2_{\AW(\theta)}\!=\!\move^3 \big)\big] \big\}$, $\EE\big[t_{\AW(\theta)} \;\big\vert\; \meet^{3}_{\AW(\theta)},\ D\big]$ equals
    \[\begin{aligned}
        & \EE\big[t_{\AW(\theta)} \;\big\vert\; \meet^{3}_{\AW(\theta)},\ \chi^1_{\AW(\theta),3}\!=\!\chi^2_{\AW(\theta),3}\!=\!\move \big] \PP\big[ \chi^1_{\AW(\theta),3}\!=\!\chi^2_{\AW(\theta),3}\!=\!\move \;\big\vert\; \meet^{3}_{\AW(\theta)},\ D \big]\\
        & \hspace{7ex} + \EE\big[t_{\AW(\theta)} \;\big\vert\; \meet^{3}_{\AW(\theta)},\ \chi^1_{\AW(\theta),3}\!\neq \!\chi^2_{\AW(\theta),3} \big] \PP\big[ \chi^1_{\AW(\theta),3}\!\neq \!\chi^2_{\AW(\theta),3} \;\big\vert\; \meet^{3}_{\AW(\theta)},\ D \big], 
    \end{aligned}\]
    because~$\meet^{3}_{\AW(\theta)}$ and~$\chi^1_{\AW(\theta),3}\!=\!\chi^2_{\AW(\theta),3}\!=\!\stay$ are disjoint events and the random variables~$\chi^i_{\AW(\theta),r}$ for $r\in \{1,2\}$ do not affect the expected meeting time when conditioning on~$\meet^{3}_{\AW(\theta)}$.
    Since $\PP\big[ \chi^1_{\AW(\theta),3}\!\neq \!\chi^2_{\AW(\theta),3} \;\big\vert\; \meet^{3}_{\AW(\theta)},\ \chi^1_{\AW(\theta)}\!=\!\chi^2_{\AW(\theta)}\!=\!\move^3 \big]=0$, this yields
    \[\begin{aligned}
        \EE\big[t&_{\AW(\theta)} \;\big\vert\; \meet^{3}_{\AW(\theta)},\ \neg \big(\chi^1_{\AW(\theta)}\!=\!\chi^2_{\AW(\theta)}\!=\!\move^3\big) \big] - \EE\big[t_{\AW(\theta)} \;\big\vert\; \meet^{3}_{\AW(\theta)},\ \chi^1_{\AW(\theta)}\!=\!\chi^2_{\AW(\theta)}\!=\!\move^3 \big]\\
        & = \big(\EE\big[t_{\AW(\theta)} \;\big\vert\; \meet^{3}_{\AW(\theta)},\ \chi^1_{\AW(\theta),3}\!\neq \!\chi^2_{\AW(\theta),3} \big] - \EE\big[t_{\AW(\theta)} \;\big\vert\; \meet^{3}_{\AW(\theta)},\ \chi^1_{\AW(\theta),3}\!=\!\chi^2_{\AW(\theta),3}\!=\!\move \big]\big)\\
        & \hspace{7ex} \cdot \PP\big[ \chi^1_{\AW(\theta),3}\!\neq \!\chi^2_{\AW(\theta),3} \;\big\vert\; \meet^{3}_{\AW(\theta)},\ \neg\big(\chi^1_{\AW(\theta)}\!=\!\chi^2_{\AW(\theta)}\!=\!\move^3 \big) \big].
    \end{aligned}\]
    The last expression is non-negative because
    \begin{equation}\label{eq:time-staying-moving}
    \begin{aligned}
        \EE\big[t_{\AW(\theta)} \;\big\vert\; \meet^{3}_{\AW(\theta)},\ \chi^1_{\AW(\theta),3}\!\neq \!\chi^2_{\AW(\theta),3} \big] - 2(n-1)&=\frac{n}{2},\\
        \EE\big[t_{\AW(\theta)} \;\big\vert\; \meet^{3}_{\AW(\theta)},\ \chi^1_{\AW(\theta),3}\!=\!\chi^2_{\AW(\theta),3}\!=\!\move \big] -2(n-1) &= n\bigg( 1-\frac{n\hat{d}^1_{n}-1}{(n-1)(1-\hat{d}^1_{n-1})} \bigg)
    \end{aligned}
    \end{equation}
    by \Cref{lem:exp-first-fixed-point}, and
    \[\begin{aligned}
        n\bigg( \frac{1}{2} - 1+&\frac{n\hat{d}^1_{n}-1}{(n-1)(1-\hat{d}^1_{n-1})} \bigg) = \frac{n}{2(n-1)(1-\hat{d}^1_{n-1})}\big(2n\hat{d}^1_{n}-n-1+(n-1)\hat{d}^1_{n-1}\big)\\
        & \hspace{3em} \geq \frac{n}{2(n-1)(1-\hat{d}^1_{n-1})} \bigg( \bigg(\frac{3}{e}-1\bigg)n + \frac{2}{e}-1 - \frac{1}{4(n-1)!} - \frac{1}{8(n-2)!}\bigg) > 0
    \end{aligned}\]
    by \Cref{lem:bounds} and because the last expression in parentheses is 
    approximately $0.22$ for~$n=5$ and trivially increasing in~$n$. This shows \eqref{ineq:et-cond-third-round}, which we will now use to complete the proof. 
    
    We claim that
    \[\begin{aligned}
        \EE &[t_{\AW(\theta)} ] -\EE[t_{\CFK(\theta)}] \\
        & = \big(\EE\big[t_{\AW(\theta)}\;\big\vert\; \meet^{3}_{\AW(\theta)}\big] - \EE\big[t_{\AW(\theta)}\;\big\vert\; \meet^{\leq 2}_{\AW(\theta)}\big] \big)\PP\big[\meet^{3}_{\AW(\theta)} \;\big\vert\; \meet^{\leq 3}_{\AW(\theta)} \big] \\
        & \hspace{5em} - \big(\EE\big[t_{\CFK(\theta)}\;\big\vert\; \meet^{3}_{\CFK(\theta)}\big] - \EE\big[t_{\AW(\theta)}\;\big\vert\; \meet^{\leq 2}_{\AW(\theta)}\big] \big)\PP\big[\meet^{3}_{\CFK(\theta)} \;\big\vert\; \meet^{\leq 3}_{\CFK(\theta)} \big] \\
        & \hspace{5em} + 3(n-1)\bigg(\frac{1}{\PP[\meet^{\leq 3}_{\AW(\theta)}]}-\frac{1}{\PP[\meet^{\leq 3}_{\CFK(\theta)}]}\bigg)\\
        & \geq \big(\EE\big[t_{\AW(\theta)}\;\big\vert\; \meet^{3}_{\AW(\theta)}\big] - \EE\big[t_{\AW(\theta)}\;\big\vert\; \meet^{\leq 2}_{\AW(\theta)}\big] \big)\big(\PP\big[\meet^{3}_{\AW(\theta)} \;\big\vert\; \meet^{\leq 3}_{\AW(\theta)} \big] - \PP\big[\meet^{3}_{\CFK(\theta)} \;\big\vert\; \meet^{\leq 3}_{\CFK(\theta)} \big]\big)\\
        & \hspace{5em} + 3(n-1)\bigg(\frac{1}{\PP[\meet^{\leq 3}_{\AW(\theta)}]}-\frac{1}{\PP[\meet^{\leq 3}_{\CFK(\theta)}]}\bigg)\\
        & = \big(\EE\big[t_{\AW(\theta)}\;\big\vert\; \meet^{3}_{\AW(\theta)}\big] - \EE\big[t_{\AW(\theta)}\;\big\vert\; \meet^{\leq 2}_{\AW(\theta)}\big] \big) \Bigg(\frac{\PP\big[\meet^{\leq 2}_{\CFK(\theta)}\big]}{\PP\big[\meet^{\leq 3}_{\CFK(\theta)} \big]} - \frac{\PP\big[\meet^{\leq 2}_{\AW(\theta)}\big]}{\PP\big[\meet^{\leq 3}_{\AW(\theta)} \big]} \Bigg)\\
        & \hspace{5em} + 3(n-1)\bigg(\frac{1}{\PP[\meet^{\leq 3}_{\AW(\theta)}]}-\frac{1}{\PP[\meet^{\leq 3}_{\CFK(\theta)}]}\bigg)\\
        & = \Big( 3(n-1) - \big(\EE\big[t_{\AW(\theta)}\;\big\vert\; \meet^{3}_{\AW(\theta)}\big] - \EE\big[t_{\AW(\theta)}\;\big\vert\; \meet^{\leq 2}_{\AW(\theta)}\big] \big)\PP\big[\meet^{\leq 2}_{\AW(\theta)}\big] \Big) 
        \bigg(\frac{1}{\PP[\meet^{\leq 3}_{\AW(\theta)}]}-\frac{1}{\PP[\meet^{\leq 3}_{\CFK(\theta)}]}\bigg).
    \end{aligned}\]
    Indeed, the first equality holds by~\eqref{eq:et-markovian} and the fact that $\EE\big[t_{\AW(\theta)}\;\big\vert\; \meet^{\leq 2}_{\AW(\theta)}\big]=\EE\big[t_{\CFK(\theta)}\;\big\vert\; \meet^{\leq 2}_{\CFK(\theta)}\big]$ by definition of the strategies, the inequality by~\eqref{ineq:et-cond-third-round}, and the last two equalities by rearranging and using that~$\PP\big[\meet^{\leq 2}_{\AW(\theta)}\big] = \PP\big[\meet^{\leq 2}_{\CFK(\theta)}\big]$.
    Since $\EE\big[t_{\AW(\theta)}\mid \meet^{3}_{\AW(\theta)}\big] \leq 2(n-1)+n/2$ by \eqref{eq:time-staying-moving}, $\EE\big[t_{\AW(\theta)}\;\big\vert\; \meet^{\leq 2}_{\AW(\theta)}\big]\geq 1$, 
    and~$\PP\big[\meet^{\leq 2}_{\AW(\theta)}\big]\leq 1$, we conclude that
    \[ 
        \EE[t_{\AW(\theta)} ] -\EE[t_{\CFK(\theta)}] \geq \frac{n-1}{2}\bigg(\frac{1}{\PP\big[\meet^{\leq 3}_{\AW(\theta)}\big]}-\frac{1}{\PP\big[\meet^{\leq 3}_{\CFK(\theta)}\big]}\bigg) \geq \frac{n-1}{2}\cdot \frac{967(1-\theta)^6}{(n-1)^9} > \frac{483(1-\theta)^6}{(n-1)^8},
    \]
    where the second inequality follows from \Cref{thm:meeting-prob} and the fact that~$x-y\geq z$ for values~$x,y,z\in (0,1)$ implies that $\frac{1}{y}-\frac{1}{x} = \frac{x-y}{xy} \geq z$. 
    This proves the result for~$\theta\in(0,1)$.

    It remains to consider~$\theta=0$.
    Since both strategies restart after~$3(n-1)$ steps,
    \[
        \EE[t_\sigma]=\frac{\EE[\min\{t_\sigma,3(n-1)\}]}{\PP[t_\sigma\leq3(n-1)]}.
    \]
    For fixed~$n$, the numerator and denominator are continuous functions of~$\theta$, since they depend only on the choices within one block. The denominator is positive at~$\theta=0$.
    Thus, the claimed equality for~$n=4$ and inequality for~$n\geq5$ also hold at~$\theta=0$ by continuity.
\end{proof}

\section{Conclusion}

We have shown that the Anderson--Weber strategy is not optimal for the symmetric rendezvous problem with any finite number $n\geq 4$ of locations, by giving an explicit improving strategy $\CFK$. The improvement is small, but it is conceptually significant: it demonstrates that carefully chosen correlations between consecutive moving rounds can outperform the independent randomization underlying $\AW$.

The improving strategy has been chosen to allow for a rigorous proof rather than to optimize the improvement, and there are a number of ways in which the meeting time could be reduced further. First, the probability $\theta$ of staying at the home location could be optimized specifically for $\CFK$ rather than being set to the value that is optimal for~$\AW$. The latter was done to make the comparison with $\AW$ easier. Second, our strategy modifies~$\AW$ only in every third round, and only when the player moves in that round and in the two preceding rounds. Since meeting occurs with probability one in any round in which exactly one player moves, the moving rounds of the two players coincide conditional on not having met. It is therefore natural to apply the same modification to every third moving round rather than to every third round. Both refinements should improve the performance for any finite $n$ but would make the analysis substantially more complicated.

Our result leaves open the broader question of finding optimal strategies when $n\geq 4$. In particular the improvement provided by $\CFK$ over $\AW$ vanishes as $n\to\infty$, and the conjecture of \citet{AW90} that $\AW$ is asymptotically optimal remains open. However, our construction suggests a possible route toward further progress. The correlations introduced by $\CFK$ can be interpreted in terms of cliques in the complement of the derangement graph, prescribing how cliques are contracted and visited. We have made certain choices to ensure analytical tractability, but these choices are not obviously optimal. A more systematic understanding of the clique structure of the complement of the derangement graph may lead to stronger correlated strategies.

A number of straightforward modifications, like optimizing $\theta$, modifying the set of affected rounds, or replacing blocks of $3$ by blocks of $n-1$ moving rounds do not appear to provide an asymptotic improvement over $\AW$. Any disproof of the conjecture of asymptotic optimality is therefore likely to require a new way of using the structure of the derangement graph. We believe that the graph-theoretic perspective we have introduced provides a useful framework for this question and may be relevant more broadly to the design of correlated strategies in symmetric rendezvous.

\newpage
\appendix

\section{Proofs Deferred from \texorpdfstring{\Cref{sec:prelims}}{Section 2}}

\subsection{Proof of \texorpdfstring{\Cref{lem:shifted-derangements}}{Lemma 1}}\label{app:lem:shifted-derangements}

We write the entries in the table as~$d_n^{n-\ell}$ for~$n \in \NN_0$ and~$\ell \in \{0,\dots,n\}$ and prove the statement by induction over~$\ell$.
For~$\ell=0$, the statement clearly holds since all permutations~$\pi \in \mathcal{P}(\{1+n,\dots,2n\})$ do not have a fixed point and there are~$d_{n}^{n-0} = n!$ of them. Suppose that the result is correct for all entries up to a fixed value of~$\ell$.
Let~$n \in \NN_0$ with~$n \geq \ell+1$ be arbitrary and consider the entry~$d_{n}^{n-\ell-1} = d_n^{n-\ell} - d_{n-1}^{n-\ell-1}$. 
For a set $S$, we let $\mathcal{D}(S)\coloneqq \{\pi\in \calP(S): \pi(i)\neq i \text{ for all } i\in [|S|]\}$ denote the set of derangements on $S$.
Note that, by the induction hypothesis, $|\mathcal{D}(\{n-\ell+1,\ldots,2n-\ell\})| = d^{n-\ell}_n$, and that this set can be partitioned into the sets
\[\begin{aligned}
\mathcal{D}' &\coloneqq \{\pi \in \mathcal{D}(\{n-\ell+1, \dots, 2n-\ell\}) : \pi(n-\ell) = 2n-\ell \},\text{ and} \\
\mathcal{D}'' &\coloneqq \{\pi \in \mathcal{D}(\{n-\ell+1, \dots, 2n-\ell\}) : \pi(n-\ell) \neq 2n-\ell \}.
\end{aligned}\]
By removing the element~$2n-\ell$, there is a bijection between~$\mathcal{D}'$ and~$\mathcal{D}(\{n-\ell,\dots,2n-\ell-2\})$, whose size is $d^{n-\ell-1}_{n-1}$ by the induction hypothesis. 
On the other hand, by identifying the element~$2n-\ell$ with the element~$n-\ell$ there is also a bijection between~$\mathcal{D}''$ and~$\mathcal{D}(\{n-\ell, \dots, 2n-\ell-1\})$.
We conclude that the size of this set is $d^{n-\ell}_{n}-d^{n-\ell-1}_{n-1} = d^{n-\ell-1}_{n}$.
This shows the result.\qed

\subsection{Proof of \texorpdfstring{\Cref{lem:bounds}}{Lemma 2}}\label{app:lem:bounds}

For~$n \in \NN$, we have
\[\begin{aligned}
d_n^0 &= n! \sum_{j=0}^n \frac{(-1)^j}{j!} = n! \Biggl(\sum_{j=0}^\infty \frac{(-1)^j}{j!} - \sum_{j=n+1}^{\infty} \frac{(-1)^j}{j!}\Biggr) = \frac{n!}{e} + \sum_{j=n+1}^\infty (-1)^{j+1} \frac{n!}{j!}.
\end{aligned}\]
If~$n$ is even, this yields~$d_n^0 > \frac{n!}{e}$ and~$d_n^0 = \lfloor \frac{n!}{e} + r \rfloor$ for~$r \in [\frac{1}{n+1},1]$.
If~$n$ is odd, we have~$d_{n}^0 < \frac{n!}{e}$ and~$d_n^0 = \lfloor \frac{n!}{e} +r \rfloor$ whenever~$0\leq r  \leq 1 - \frac{1}{n+1}$. Both inequalities are satisfied for~$r \in \bigl[\frac{1}{3},\frac{1}{2}]$.

For~$n \in \NN$, we further have
\[\begin{aligned}
d_n^1 = d_n^0 + d_{n-1}^0 & = n! \sum_{j=0}^n \frac{(-1)^j}{j!} + (n-1)! \sum_{j=0}^{n-1} \frac{(-1)^j}{j!} \\
&= \frac{n!}{e} + \sum_{j=n+1}^\infty (-1)^{j+1} \frac{n!}{j!} + \frac{(n-1)!}{e} + \sum_{j=n}^\infty (-1)^{j+1} \frac{(n-1)!}{j!} \\
&= \frac{n! + (n-1)!}{e} + (-1)^{n+1}\frac{1}{n} + \sum_{j=n+1}^\infty (-1)^{j+1} \frac{n! + (n-1)!}{j!}.
\end{aligned}\]
Using that~$\frac{n! + (n-1)!}{(n+1)!} = \frac{1}{n}$ this simplifies to
\[\begin{aligned}
d_n^1 &= \frac{n! + (n-1)!}{e} + \sum_{j=n+2}^{\infty} (-1)^{j+1} \frac{n! + (n-1)!}{j!}.
\end{aligned}\]
If~$n$ is odd, this yields~$d_n^1 > \frac{n! + (n-1)!}{e}$ and~$d_n^1 = \lfloor \frac{n! + (n-1)!}{e} + r\rfloor$ for~$r \in [\frac{1}{(n+1)(n+2)} + \frac{1}{n(n+1)(n+2)}, 1]$.  If~$n$ is even, we have~$d_n^1 < \frac{n! + (n-1)!}{e}$ and~$d_n^1 = \lfloor \frac{n! + (n-1)!}{e} + r \rfloor$ for~$0\leq r \leq 1 - \frac{1}{(n+1)(n+2)} - \frac{1}{n(n+1)(n+2)}$. Both inequalities are satisfied for~$r \in \bigl[ \frac{1}{3}, \frac{7}{8} \bigr]$.

For~$n \in \NN$ with~$n \geq 2$, we further have
\begingroup
\allowdisplaybreaks
\[\begin{aligned}
d_n^2 &= d_n^1 + d_{n-1}^1 \\
&= d_n^0 + 2d_{n-1}^0 + d_{n-2}^0 \\
&= n! \sum_{j=0}^n \frac{(-1)^j}{j!} + 2(n-1)!\sum_{j=0}^{n-1} \frac{(-1)^j}{j!} + (n-2)!\sum_{j=0}^{n-2}\frac{(-1)^j}{j!} \\
&= \frac{n! + 2(n-1)! + (n-2)!}{e} + \sum_{j=n+1}^{\infty} (-1)^{j+1} \frac{n!}{j!} +  2\sum_{j=n}^{\infty} (-1)^{j+1} \frac{(n-1)!}{j!} +   \sum_{j=n-1}^{\infty} (-1)^{j+1} \frac{(n-2)!}{j!} \\
&= \frac{n! + 2(n-1)! + (n-2)!}{e} + (-1)^n\frac{1}{n-1} + (-1)^{n+1} \biggl(\frac{2}{n} + \frac{1}{n(n-1)}\biggr) \\
&\quad + \sum_{j=n+1}^\infty (-1)^{j+1}\frac{n!+2(n-1)!+(n-2)!}{j!}.
\end{aligned}\]
\endgroup
It is straightforward to check that
\[\begin{aligned}
\frac{1}{n+1} + \frac{2}{n(n+1)} + \frac{1}{(n-1)n(n+1)} + \frac{1}{n-1} - \biggl(\frac{2}{n} + \frac{1}{n(n-1)}\biggr) = \frac{1}{(n+1)(n-1)} \enspace.
\end{aligned}\]
Using this equality, we obtain
\[\begin{aligned}
d^2_n = \frac{n! + 2(n-1)! + (n-2)!}{e} + (-1)^n \frac{1}{(n+1)(n-1)} + \sum_{j=n+2}^\infty (-1)^{j+1} \frac{n! + 2(n-1)! + (n-2)!}{j!}.
\end{aligned}\]
It can be checked that for~$n \geq 2$ we have that
\[\begin{aligned}
\frac{1}{(n+1)(n-1)} > \frac{1}{(n+1)(n+2)} + \frac{2}{n(n+1)(n+2)} + \frac{1}{(n-1)n(n+1)(n+2)}.
\end{aligned}\]
As a consequence, we obtain that if~$n$ is even, we have~$d_n^2 > \frac{n! + 2(n-1)! + (n-2)!}{e}$ and
\[\begin{aligned}
d_n^2 = \bigg\lfloor \frac{n! + 2(n-1)! + (n-2)!}{e} + r \bigg\rfloor
\end{aligned}\]
for~$r \in \bigl[ \frac{1}{(n+1)(n-1)}, 1\bigr]$. If~$n$ is odd, we have~$d_n^2 < \frac{n! + 2(n-1)! + (n-2)!}{e}$ and the formula is satisfied whenever~$0\leq r \leq 1 - \frac{1}{(n+1)(n-1)}$. Since~$n \geq 2$, both formulas are satisfied for~$r \in \bigl[ \frac{1}{3},\frac{7}{8}\bigr]$.\qed

\section{Proofs Deferred from \texorpdfstring{\Cref{sec:warm-up}}{Section 3}}

\subsection{Proof of \texorpdfstring{\Cref{prop:simple-aw}}{Proposition 1}}\label{app:prop:simple-aw}

\propSimpleAW*

\begin{proof}
    Let~$G$ be a graph as in the statement. 
    For any~$r\in\NN$, the probability of not meeting after~$r$ steps with the uniform strategy is~$\delta^r$.
    For the clique strategies, the players meet for sure whenever they choose vertices in the same clique.
    If they choose vertices independently and uniformly from distinct cliques, \eqref{eq:uniform-clique-degrees} gives non-meeting probability $\frac{q-\delta_{\mathrm{out}}}{q} =\frac{m}{m-1}\delta>0$, independently of which distinct cliques they choose.
    
    To show the first inequality in the statement, we observe that
    \begin{equation}
        \PP[t_{\Vis_1(\mu)}> 1] = \frac{m}{m-1}\delta \bigg[\mu^2 + 2\mu(1-\mu) \frac{m-2}{m-1} + (1-\mu)^2 \bigg(\frac{1}{m-1}+\bigg(\frac{m-2}{m-1}\bigg)^2\bigg)\bigg],\label{eq:visibility-1round}
    \end{equation}
    because the players meet for sure if they visit the same clique in step~$1$, and do not meet with probability~$\frac{m}{m-1}\delta$ if they visit different cliques in step~$1$, which occurs
    \begin{enumerate}[label=(\roman*)]
        \item with probability~$1$ if both stay in the clique where they started;
        \item with probability~$1-1/(m-1)$ if exactly one of them moves to another clique, because this is the other player's clique with probability~$1/(m-1)$; and
        \item with probability~$\frac{1}{m-1}+\big(\frac{m-2}{m-1}\big)^2$ if both move to another clique, because they choose different cliques with probability~$1$ when one player visits the other player's former clique and with probability~$(m-2)/(m-1)$ otherwise.
    \end{enumerate}
    The expression on the right-hand side of \eqref{eq:visibility-1round} is not larger than~$\PP[t_{\Uni}>1]=\delta$ if and only if
    \begin{align*}
        \frac{m}{m-1} \bigg[ \mu^2 + & 2\mu(1-\mu) \frac{m-2}{m-1} + (1-\mu)^2 \bigg(\frac{1}{m-1}+\bigg(\frac{m-2}{m-1}\bigg)^2 \bigg)\bigg] \leq 1\\
        \Longleftrightarrow\ & m(m-1)^2\mu^2 + 2m(m-1)(m-2)\mu(1-\mu) + m(m^2-3m+3)(1-\mu)^2 \leq (m-1)^3\\
        \Longleftrightarrow\ & m\big((m-1)^2-2(m-1)(m-2)+(m^2-3m+3)\big) \mu^2 \\
        & + 2m\big((m-1)(m-2)-(m^2-3m+3)\big)\mu + m(m^2-3m+3) - (m-1)^3 \leq 0 \\
        \Longleftrightarrow\ & (m\mu - 1)^2 \leq 0,
    \end{align*}
    which holds with equality if and only if~$\mu=1/m$.
    For any other value of~$\mu$, we have the expression on the right-hand side of \eqref{eq:visibility-1round} is strictly larger than~$\delta$.
    This finishes the proof of the first inequality in the statement.

    For the second inequality, it suffices to prove that~$\PP[t_{\Vis_2(1/m)}>2] < \delta^2$.
    Indeed, if this holds, then for every~$k\in \NN$ and~$r\in \{0,1\}$ we have
    \[
        \PP\big[t_{\Vis_2(1/m)}>2k+r\big] = \PP\big[t_{\Vis_2(1/m)}>2\big]^k \delta^r < \delta^{2k+r} = \PP\big[t_{\Uni}>2k+r\big],
    \]
    where the first equality holds because~$\Vis_2(1/m)$ restarts and repeats after every pair of steps.
    Thus, we focus on proving that~$\PP[t_{\Vis_2(1/m)}> 2] < \delta^2$ holds in what follows.
    
    We first compute the probability that the players choose different cliques in each of the two steps under~$\Vis_2(\mu)$, conditional on both players moving.
    By conditioning on the number of players visiting the home clique of the other in the first step, we obtain that this probability equals
    \begin{align*}
        \bigg(\frac{1}{m-1}\bigg)^2 & \frac{m-3}{m-2} + \frac{2(m-2)}{(m-1)^2}\bigg( \frac{1}{m-2}+\bigg(\frac{m-3}{m-2}\bigg)^2\bigg)\\
        & + \frac{(m-2)(m-3)}{(m-1)^2} \bigg(\frac{2}{m-2}+\frac{(m-4)(m-3)}{(m-2)^2}\bigg) = \frac{m^3-6m^2+14m-13}{(m-1)^2(m-2)}.
    \end{align*}
    Indeed, the first term corresponds to both players visiting the other's home clique in the first step.
    This event occurs with probability~$1/(m-1)^2$ and, conditional on it, both players choose a clique uniformly at random in the second step from the same set of~$m-2$ cliques, so they choose different ones with probability~$(m-3)/(m-2)$.
    The second term corresponds to exactly one player visiting the other's home clique in the first step, which occurs with probability~$2(m-2)/(m-1)^2$.
    Conditional on this event, without loss of generality, let us assume that player~$1$ visits the home clique of player~$2$ in the first step, the players choose different cliques in the second step if either player~$2$ visits the home clique of player~$1$, which occurs with probability~$1/(m-2)$, or both visit different cliques that are not the home clique of any of them, which occurs with probability~$(m-3)^2/(m-2)^2$.
    The third term corresponds to the players visiting different cliques that are not the home clique of any of them in the first step, which occurs with probability~$(m-2)(m-3)/(m-1)^2$ and only when $m\geq 4$.
    Conditional on this event, the players choose different cliques in the second step if either player~$1$ visits a clique that player~$2$ already visited, which occurs with probability~$2/(m-2)$, or player~$1$ visits another clique but player~$2$ does not choose the same one, which occurs with probability~$(m-3)(m-4)/(m-2)^2$.

    We can now compute the non-meeting probability~$\Vis_2(\mu)$ after two steps similarly to the case before, by conditioning on the number of players that move across cliques:
    \begin{equation*}
        \PP[t_{\Vis_2(\mu)}> 2] = \bigg(\frac{m}{m-1}\bigg)^2\delta^2 \bigg[\mu^2 + 2\mu(1-\mu) \frac{m-3}{m-1} + (1-\mu)^2 \frac{m^3-6m^2+14m-13}{(m-1)^2(m-2)}\bigg].
    \end{equation*}
    This expression is strictly smaller than~$\PP[t_{\Uni}>2]=\delta^2$ if and only if
    \begin{align*}
        \bigg(\frac{m}{m-1}\bigg)^2 \bigg[\mu^2 & + 2\mu(1-\mu) \frac{m-3}{m-1} + (1-\mu)^2 \frac{m^3-6m^2+14m-13}{(m-1)^2(m-2)}\bigg] < 1\\
        \Longleftrightarrow\ & m^2(m-1)^2(m-2)\mu^2 + 2m^2(m-1)(m-2)(m-3)\mu(1-\mu) \\
        & + m^2(m^3-6m^2+14m-13)(1-\mu)^2 < (m-1)^4(m-2)\\
        \Longleftrightarrow\ & m^2\big((m-1)^2(m-2) - 2(m-1)(m-2)(m-3) + (m^3-6m^2+14m-13) \big) \mu^2\\
        & +2m^2 \big( (m-1)(m-2)(m-3) - (m^3-6m^2+14m-13) \big) \mu \\
        & + m^2(m^3-6m^2+14m-13) -  (m-1)^4(m-2) < 0 \\
        \Longleftrightarrow\ & m^2(2m^2-3m-3)\mu^2 -2m^2(3m-7)\mu + 3m^2-9m+2 < 0.
    \end{align*}
    Taking~$\mu=1/m$, the left-hand side becomes 
    \[
    2m^2-3m-3-6m^2+14m+3m^2-9m+2 = -(m-1)^2,
    \]
    which is strictly negative for any~$m>1$.
\end{proof}

\section{Proofs Deferred from \texorpdfstring{\Cref{subsec:meeting-prob}}{Section 4.1}}

\subsection{Proof of \texorpdfstring{\Cref{lem:cond-prob-formula}}{Lemma 3}}\label{app:lem:cond-prob-formula}

    We first fix~$n\in\NN$ with~$n\geq5$ and apply the law of
    total probability to obtain
    \begin{equation}\label{eq:total-probs}
        \PP\big[\bar{\meet}^{\leq3}_{\CFK(\theta)} \;\big\vert\; \chi^1_{\CFK(\theta)}\!=\!\chi^2_{\CFK(\theta)}\! =\!\move^3 \big] =  \sum_{k,\ell\in [4]} \PP\big[\bar{\meet}^{\leq3}_{\CFK(\theta)} \;\big\vert\; \chi^1_{\CFK(\theta)}\!=\!\chi^2_{\CFK(\theta)}\!=\!\move^3 ,\ A^1_{k},\ A^2_{\ell} \big] \cdot w^1_{k} \cdot w^2_{\ell}.
    \end{equation}
    We denote the first term in the sum by \(Z_{k\ell}\coloneqq \PP\big[\bar{\meet}^{\leq3}_{\CFK(\theta)} \;\big\vert\; \chi^1_{\CFK(\theta)}\!=\!\chi^2_{\CFK(\theta)}\!=\!\move^3 ,\ A^1_{k},\ A^2_{\ell} \big],\) corresponding to the non-meeting probability under~$\CFK(\theta)$ conditional on both players moving in the first three rounds and on the events~$A^1_{k}$ and~$A^2_{\ell}$, for~$k,\ell\in [4]$.
    It is clear that~$Z\in [0,1]^{4\times 4}$ is a symmetric matrix; we now compute its entries on and above the diagonal.
    We repeatedly use \Cref{lem:shifted-derangements} to compute the non-meeting probabilities in steps~$2,\ldots,n-1$ of each round conditional on the first locations: if these first locations are each other's home locations, the corresponding non-meeting probability is~$\hat{d}^0_{n-2}$, if exactly one of these first locations is the other player's home location, the corresponding non-meeting probability is~$\hat{d}^1_{n-2}$; if none of these first locations are each other's home locations and they are different, the corresponding non-meeting probability is~$\hat{d}^2_{n-2}$.

    When we condition on~$A^1_{1}$, we have~$\pi^1_{\CFK(\theta),r}(1)=2$ and thus~$\pi^1_{\CFK(\theta),r}(1)\neq \pi^2_{\CFK(\theta),r}(1)$ for each~$r\in [3]$.
    If~$\pi^2_{\CFK(\theta),r}(1)=1$ for a round~$r\in [3]$, the meeting probability in later steps of that round (between~$(r-1)(n-1)+2$ and~$r(n-1)$) is~$1-\hat{d}^0_{n-2}$.
    Similarly, if~$\pi^2_{\CFK(\theta),r}(1)\neq 1$ for a round~$r\in [3]$, the meeting probability in later steps of that round is~$1-\hat{d}^1_{n-2}$.
    Therefore,
    \[
        Z_{11} = (\hat{d}^0_{n-2})^3,\qquad
        Z_{12} = (\hat{d}^1_{n-2})^3,\qquad
        Z_{13} = \hat{d}^0_{n-2} (\hat{d}^1_{n-2})^2,\qquad 
        Z_{14} = (\hat{d}^1_{n-2})^3.
    \]

    When we condition on~$A^1_{2}$, we have~$\pi^1_{\CFK(\theta),r}(1)\neq 2$ for all~$r\in [3]$, so the event~$\pi^1_{\CFK(\theta),r}(1)= \pi^2_{\CFK(\theta),r}(1)$ for some~$r\in [3]$ has non-zero probability.
    Indeed, it occurs with probability~$1/(n-2)$ conditional on~$A^2_{2}$, with probability~$2/(n-2)$ conditional on~$A^2_{3}$, and with probability~$3/(n-2)$ conditional on~$A^2_{4}$.
    In addition, if~$\pi^2_{\CFK(\theta),r}(1)=1$ for a round~$r\in [3]$, the meeting probability in later steps of that round is~$1-\hat{d}^1_{n-2}$.
    If~$\pi^2_{\CFK(\theta),r}(1)\neq 1$ for a round~$r\in [3]$, the meeting probability in later steps of that round, conditional on~$\pi^1_{\CFK(\theta),r}(1)\neq \pi^2_{\CFK(\theta),r}(1)$, is~$1-\hat{d}^2_{n-2}$.
    Therefore,
    \[
        Z_{22} = \frac{n-3}{n-2} (\hat{d}^2_{n-2})^3,\qquad
        Z_{23} = \frac{n-4}{n-2} \hat{d}^1_{n-2} (\hat{d}^2_{n-2})^2,\qquad
        Z_{24} = \frac{n-5}{n-2} (\hat{d}^2_{n-2})^3.
    \]

    When we condition on~$A^1_{3}$ and~$A^2_{3}$, we have~$\pi^1_{\CFK(\theta),r}(1)=2$ and~$\pi^2_{\CFK(\theta),s}(1)=1$ for some rounds~$r,s\in [3]$, so we need to further condition on whether~$r=s$, which occurs with probability~$1/3$, or~$r\neq s$, which occurs with probability~$2/3$.
    If~$r=s$, we have~$\pi^1_{\CFK(\theta),r}(1)\neq \pi^2_{\CFK(\theta),r}(1)$ and the meeting probability in later steps of that round is~$1-\hat{d}^0_{n-2}$.
    The event~$\pi^1_{\CFK(\theta),r'}(1)= \pi^2_{\CFK(\theta),r'}(1)$ for some other round~$r'\in [3]\setminus \{r\}$ occurs with probability~$\frac{1}{n-2}+\frac{n-4}{n-2}\cdot \frac{1}{n-3}$.
    The meeting probability in later steps of a round~$r'\in [3]\setminus \{r\}$, conditional on~$\pi^1_{\CFK(\theta),r'}(1)\neq \pi^2_{\CFK(\theta),r'}(1)$, is~$1-\hat{d}^2_{n-2}$.
    If~$r\neq s$, we have~$\pi^1_{\CFK(\theta),r'}(1)\neq \pi^2_{\CFK(\theta),r'}(1)$ for~$r'\in \{r,s\}$, and the meeting probability in later steps of these rounds is~$1-\hat{d}^1_{n-2}$.
    The event~$\pi^1_{\CFK(\theta),r'}(1)= \pi^2_{\CFK(\theta),r'}(1)$ for the round~$r'\in [3]\setminus \{r,s\}$ occurs with probability~$\frac{1}{n-2}$.
    The meeting probability in later steps of the round~$r'\in [3]\setminus \{r,s\}$, conditional on~$\pi^1_{\CFK(\theta),r'}(1)\neq \pi^2_{\CFK(\theta),r'}(1)$, is~$1-\hat{d}^2_{n-2}$.
    Combining these probabilities, we obtain
    \[
        Z_{33} = \frac{1}{3} \bigg( 1-\bigg(\frac{1}{n-2}+\frac{n-4}{n-2}\cdot \frac{1}{n-3}\bigg)\bigg)  \hat{d}^0_{n-2} (\hat{d}^2_{n-2})^2 + \frac{2}{3} \bigg( 1-\frac{1}{n-2}\bigg) (\hat{d}^1_{n-2})^2 \hat{d}^2_{n-2}.
    \]

    When we condition on~$A^1_{3}$ and~$A^2_{4}$, we have~$\pi^1_{\CFK(\theta),r}(1)=2$ for some round~$r\in [3]$.
    This implies~$\pi^1_{\CFK(\theta),r}(1)\neq \pi^2_{\CFK(\theta),r}(1)$ and the meeting probability in later steps of that round is~$1-\hat{d}^1_{n-2}$.
    The event~$\pi^1_{\CFK(\theta),s}(1)= \pi^2_{\CFK(\theta),s}(1)$ for some other round~$s\in [3]\setminus \{r\}$ occurs with probability~$\frac{1}{n-2}+\frac{n-4}{n-2}\cdot \frac{1}{n-3}$.
    The meeting probability in later steps of a round~$s\in [3]\setminus \{r\}$, conditional on~$\pi^1_{\CFK(\theta),s}(1)\neq \pi^2_{\CFK(\theta),s}(1)$, is~$1-\hat{d}^2_{n-2}$.
    Therefore,
    \[
        Z_{34} = \bigg( 1-\bigg(\frac{1}{n-2}+\frac{n-4}{n-2}\cdot \frac{1}{n-3}\bigg)\bigg)  \hat{d}^1_{n-2} (\hat{d}^2_{n-2})^2.
    \]

    Finally, when we condition on~$A^1_{4}$ and~$A^2_{4}$, we compute the probability that~$\pi^1_{\CFK(\theta),r}(1)\neq \pi^2_{\CFK(\theta),r}(1)$ for every round~$r\in [3]$ by distinguishing whether the events~$\pi^2_{\CFK(\theta),1}(1)=\pi^1_{\CFK(\theta),1}(1)$ or~$\pi^2_{\CFK(\theta),1}(1)\in \{\pi^1_{\CFK(\theta),2}(1),\pi^1_{\CFK(\theta),3}(1)\}$ hold.
    If~$\pi^2_{\CFK(\theta),1}(1)=\pi^1_{\CFK(\theta),1}(1)$, the non-meeting probability is trivially zero.
    If ~$\pi^2_{\CFK(\theta),1}(1)\in \{\pi^1_{\CFK(\theta),2}(1),\pi^1_{\CFK(\theta),3}(1)\}$ holds, which occurs with probability~$2/(n-2)$, we suppose w.l.o.g.~$\pi^2_{\CFK(\theta),1}(1)=\pi^1_{\CFK(\theta),2}(1)$. Then,~$\pi^1_{\CFK(\theta),r}(1)\neq \pi^2_{\CFK(\theta),r}(1)$ for~$r\in [2]$, and we have~$\pi^1_{\CFK(\theta),r}(1)\neq \pi^2_{\CFK(\theta),r}(1)$ for every round~$r\in [3]$ whenever~$\pi^1_{\CFK(\theta),3}(1)\neq \pi^2_{\CFK(\theta),3}(1)$, which occurs with probability~$(n-4)/(n-3)$ conditional on the former event.
    Finally, if~$\pi^2_{\CFK(\theta),1}(1)\notin \{\pi^1_{\CFK(\theta),1}(1),\pi^1_{\CFK(\theta),2}(1),\pi^1_{\CFK(\theta),3}(1)\}$, we have~$\pi^1_{\CFK(\theta),r}(1)\neq \pi^2_{\CFK(\theta),r}(1)$ for every round~$r\in [3]$ if either~$\pi^2_{\CFK(\theta),2}(1)=\pi^1_{\CFK(\theta),3}(1)$ holds, which occurs with probability~$1/(n-3)$, or if both~$\pi^2_{\CFK(\theta),2}(1)\notin \{\pi^1_{\CFK(\theta),2}(1),\pi^1_{\CFK(\theta),3}(1)\}$ and~$\pi^2_{\CFK(\theta),3}(1)\neq \pi^1_{\CFK(\theta),3}(1)$ hold, which occurs with overall probability~$\frac{n-5}{n-3} \cdot \frac{n-5}{n-4}$.
    Regarding the later steps of each round, conditional on having~$\pi^1_{\CFK(\theta),r}(1)\neq \pi^2_{\CFK(\theta),r}(1)$ for every round~$r\in [3]$, the meeting probability in later steps of each of these rounds is~$1-\hat{d}^2_{n-2}$.
    Therefore,
    \[\begin{aligned}
        Z_{44} & = \bigg( \frac{n-5}{n-2}\bigg(\frac{n-5}{n-3}\cdot \frac{n-5}{n-4} + \frac{1}{n-3}\bigg) +\frac{2}{n-2}\cdot \frac{n-4}{n-3}\bigg) (\hat{d}^2_{n-2})^3.
    \end{aligned}\]

Substituting the previous expressions and the probabilities from \eqref{eq:probs-first-locs} into \eqref{eq:total-probs} yields the expression in the statement.

It remains to consider~$n=4$.
In this case, $w^i_4=0$ for~$i\in\{1,2\}$, so events involving~$A^i_4$ must be omitted from the decomposition.
We therefore sum only over~$k,\ell\in[3]$.
The calculations above for these indices remain valid: in particular, conditional on~$A^1_2$ and~$A^2_3$, the players necessarily start at the same location in one of the three rounds, so~$Z_{23}=0$.
The expressions for~$Z_{k\ell}$ with~$k,\ell\in[3]$ then give the matrix
\[
    Z= \begin{pmatrix}
        (\hat{d}^0_2)^3 & (\hat{d}^1_2)^3 & \hat{d}^0_2 (\hat{d}^1_2)^2 \\
        (\hat{d}^1_2)^3 & \frac12 (\hat{d}^2_2)^3 & 0 \\
        \hat{d}^0_2 (\hat{d}^1_2)^2 & 0 &
        \frac16 \hat{d}^0_2 (\hat{d}^2_2)^2+\frac13 (\hat{d}^1_2)^2(\hat{d}^2_2)
    \end{pmatrix}.
\]
By \eqref{eq:probs-first-locs},
$(w^i_1,w^i_2,w^i_3)=(1/9,2/9,2/3)$
for~$i\in\{1,2\}$.
Weighting the entries of this matrix by~$w^1_kw^2_\ell$
and summing therefore yields
\(
    \frac{1}{81}
    \big(
        (\hat{d}^0_2)^3+12 \hat{d}^0_2 (\hat{d}^1_2)^2+6 \hat{d}^0_2 (\hat{d}^2_2)^2
        +4(\hat{d}^1_2)^3+12(\hat{d}^1_2)^2 \hat{d}^2_2 +2(\hat{d}^2_2)^3
    \big).
\)
This is precisely the polynomial obtained from the expression in the statement by setting~$n=4$.
Substituting~$\hat d^0_2=1/2$, $\hat d^1_2=1/2$, and~$\hat d^2_2=1$ gives the non-meeting probability~$1/8$.
Thus, both the formula and the stated value hold also for~$n=4$.

For~$n\in \{5,6,7\}$, we evaluate the expression explicitly to obtain the values in the statement. 
For~$n=5$, replacing $\hat{d}^1_4 = \frac{11}{24}$, $\hat{d}^0_3 = \frac{1}{3}$, $\hat{d}^1_3 =\frac{1}{2}$, and $\hat{d}^2_3 = \frac{2}{3}$ yields $\PP\big[\bar{\meet}^{\leq3}_{\CFK(\theta)} \;\big\vert\; \chi^1_{\CFK(\theta)}\!=\!\chi^2_{\CFK(\theta)}\!=\!\move^3 \big] = \frac{5}{54}$. 
For~$n=6$, replacing $\hat{d}^1_5 = \frac{53}{120}$, $\hat{d}^0_4 = \frac{3}{8}$, $\hat{d}^1_4 = \frac{11}{24}$, and $\hat{d}^2_4 =\frac{7}{12}$ yields $\PP\big[\bar{\meet}^{\leq3}_{\CFK(\theta)} \;\big\vert\; \chi^1_{\CFK(\theta)}\!=\!\chi^2_{\CFK(\theta)} \!=\!\move^3 \big] = \frac{439471}{5184000}$. 
Finally, for~$n=7$, replacing $\hat{d}^1_6 =\frac{103}{240}$, $\hat{d}^0_5 = \frac{11}{30}$, $\hat{d}^1_5 = \frac{53}{120}$, and $\hat{d}^2_5 =\frac{8}{15}$ yields $\PP\big[\bar{\meet}^{\leq3}_{\CFK(\theta)} \;\big\vert\; \chi^1_{\CFK(\theta)}\!=\!\chi^2_{\CFK(\theta)}\!=\!\move^3 \big] = \frac{406417}{5184000}$. \qed

\subsection{Proof of \texorpdfstring{\Cref{lem:improvement-prob}}{Lemma 4}}\label{app:lem:improvement-prob}

Before proving the lemma, we let~$r_1,r_2,r_3,r_4\colon\{m\in\NN:m\geq8\}\to\RR$ be the functions defined by
    \[\begin{aligned}
    r_1(m) & = \frac{8m^7-88m^6+360m^5-664m^4+520m^3-168m^2+184m-152}{(m-2)e^3},\\
    r_2(m) & = \frac{512m^5-5120m^4+21504m^3-49125m^2+62868m-36271}{1728(m-2)^2(m-4)!^3},\\
    r_3(m) & = \frac{64m^6-704m^5+3127m^4-7233m^3+9081m^2-5103m+128}{24(m-2)(m-4)!^2e},\\
    r_4(m) & = \frac{8m^7-93m^6+426m^5-980m^4+1182m^3-631m^2-8m-24}{(m-4)!e^2}.
    \end{aligned}\]
\begin{lemma}\label{claim:monotone-terms}
    The function~$r_1$ is increasing and the functions~$r_2,r_3,r_4$ are decreasing for~$m\geq 8$.
\end{lemma}
\begin{proof}
    For every integer~$s\geq0$, we compute the differences between consecutive terms:
    \[\begin{aligned}
        r_1(s+9)-r_1(s+8) = \frac{1}{(s+6)(s+7)e^3}\big( & 48s^7+2304s^6+47200s^5+534928s^4+3621952s^3\\[-10pt]
        & +14650896s^2+32780720s+31295720\big),\\
    \end{aligned}\]
    \[\begin{aligned}
        r_2(s+8)-r_2(s+9) =  \frac{(s+6)(s+7)}{1728(s+7)!^3} \big( & 512s^{10}+30208s^9+801792s^8+12604955s^7 \\[-10pt] & +129949267s^6 +917746079s^5+4495276316s^4 \\[-5pt] &+15074227123s^3+33106194437s^2 \\[-5pt] &+42978289645s+25028570525\big),\\
        r_3(s+8)-r_3(s+9) =  \frac{1}{24(s+7)!^2e} \big(& 64s^{11}+4288s^{10}+130295s^9+2369553s^8 \\[-10pt]
        &+28647990s^7+241675421s^6+1450928206s^5 \\[-5pt]
        &+6195440867s^4+18424957335s^3+36310894223s^2 \\[-5pt]
        &+42624389742s+22540901040\big),\\
        r_4(s+8)-r_4(s+9) = \frac{1}{{(s+5)!e^2}} \big(& 8s^8+387s^7+8078s^6+94698s^5+678407s^4 \\[-10pt]
        &+3016649s^3+8019825s^2+11352160s+6162816\big). 
    \end{aligned}\]
    Every polynomial in parentheses has strictly positive coefficients, and every factor is positive for all integers~$s\geq0$.
    Thus, all four differences are positive, which proves the claimed monotonicity.
\end{proof}

\begin{proof}[Proof of \Cref{lem:improvement-prob}]
Combining \Cref{lem:cond-prob-formula} and the bounds
\[
    \hat{d}^0_{m} \leq \frac{1}{e} + \frac{1}{3m!},\quad 
    \frac{1}{e} + \frac{1}{me} - \frac{1}{8m!} \leq \hat{d}^1_{m}\leq \frac{1}{e} + \frac{1}{me} + \frac{1}{3m!},\quad
    \hat{d}^2_{m} \leq \frac{1}{e} + \frac{2}{me} + \frac{1}{m(m-1)e} + \frac{1}{3m!}
\]
that hold for all~$m\in \NN$ due to \Cref{lem:bounds}, we obtain
\[\begin{aligned}
    (\hat{d}&^1_{n-1})^3  - \PP\big[ \bar{\meet}^{\leq 3}_{\CFK(\theta)} \;\big\vert\; \chi^1_{\CFK(\theta)}\!=\!\chi^2_{\CFK(\theta)}\!=\!\move^3 \big]\\
    \geq & \bigg( \frac{1}{e} + \frac{1}{(n-1)e} - \frac{1}{8(n-1)!} \bigg)^3 - \frac{1}{(n-1)^4} \bigg(\frac{1}{e} + \frac{1}{3(n-2)!}\bigg)^3\\
    & -\frac{6(n-2)}{(n-1)^4} \bigg(\frac{1}{e} + \frac{1}{3(n-2)!}\bigg) \bigg( \frac{1}{e} + \frac{1}{(n-2)e} + \frac{1}{3(n-2)!}\bigg)^2\\
    & - \frac{3(n-2)(n^2-7n+13)}{(n-1)^4(n-3)} \bigg(\frac{1}{e} + \frac{1}{3(n-2)!}\bigg)  \bigg( \frac{1}{e} + \frac{2}{(n-2)e} + \frac{1}{(n-2)(n-3)e} + \frac{1}{3(n-2)!}\bigg)^2\\
    & -\frac{2(n-2)(n-3)}{(n-1)^4} \bigg( \frac{1}{e} + \frac{1}{(n-2)e} + \frac{1}{3(n-2)!}\bigg)^3 \\
    & - \frac{6(n-2)(n-3)}{(n-1)^4} \bigg( \frac{1}{e} + \frac{1}{(n-2)e} + \frac{1}{3(n-2)!}\bigg)^2  \bigg( \frac{1}{e} + \frac{2}{(n-2)e} + \frac{1}{(n-2)(n-3)e} + \frac{1}{3(n-2)!}\bigg)\\
    & -\frac{6(n-2)(n-4)(n^2-6n+10)}{(n-1)^4(n-3)} \bigg( \frac{1}{e} + \frac{1}{(n-2)e} + \frac{1}{3(n-2)!}\bigg) \\ & \hspace{5em} \bigg( \frac{1}{e} + \frac{2}{(n-2)e} + \frac{1}{(n-2)(n-3)e} + \frac{1}{3(n-2)!}\bigg)^2\\
    & -\frac{n^5-16n^4+103n^3-335n^2+551n-362}{(n-1)^4(n-3)} \bigg( \frac{1}{e} + \frac{2}{(n-2)e} + \frac{1}{(n-2)(n-3)e} + \frac{1}{3(n-2)!}\bigg)^3\\
    ={} & \frac{r_1(n)-r_2(n)-r_3(n)-r_4(n)}{8(n-1)^4(n-2)(n-3)^4}.
\end{aligned}\]

The numerator is greater than~$8831$ when~$n=8$ and, by \Cref{claim:monotone-terms}, increasing in~$n$ when~$n\geq 8$. 
Thus, for~$n\geq 8$,
\[
(\hat{d}^1_{n-1})^3  - \PP\big[\bar{\meet}^{\leq 3}_{\CFK(\theta)} \;\big\vert\; \chi^1_{\CFK(\theta)}\!=\!\chi^2_{\CFK(\theta)}\!=\!\move^3 \big] \geq \frac{8831}{8(n-1)^4(n-2)(n-3)^4} > \frac{967}{(n-1)^9}.
\]
This completes the proof for~$n\geq 8$.

For~$4\leq n\leq 7$ we can compute both of the terms in the statement exactly. It is easily verified that~$\hat{d}^1_3 = \frac{3}{6} = \frac{1}{2}$,~$\hat{d}^1_4 = \frac{11}{24}$,~$\hat{d}^1_5 = \frac{53}{120}$, and~$\hat{d}^1_6 = \frac{309}{720}$, which gives us the value of the first term. The value of the second term is given in \Cref{lem:cond-prob-formula}. Thus, for~$n=4$, 
\[
(\hat{d}^1_3)^3 - \PP\big[\bar{\meet}^{\leq 3}_{\CFK(\theta)} \;\big\vert\; \chi^1_{\CFK(\theta)}\!=\!\chi^2_{\CFK(\theta)}\!=\!\move^3 \big] = \frac{1}{8}-\frac{1}{8} = 0.
\]
For~$n=5$,
\[
(\hat{d}^1_4)^3 - \PP\big[\bar{\meet}^{\leq 3}_{\CFK(\theta)} \;\big\vert\; \chi^1_{\CFK(\theta)}\!=\!\chi^2_{\CFK(\theta)}\!=\!\move^3 \big] = \frac{1331}{13824}-\frac{5}{54} =\frac{51}{13824} = \frac{17}{4608} > 
\frac{967}{4^9}.
\]
For~$n=6$,
\[\begin{aligned}
(\hat{d}^1_5)^3 - \PP\big[\bar{\meet}^{\leq 3}_{\CFK(\theta)} \;\big\vert\; \chi^1_{\CFK(\theta)}\!=\!\chi^2_{\CFK(\theta)}\!=\!\move^3 \big] & = \frac{148877}{1728000}-\frac{439471}{5184000} = \frac{179}{129600} > 
\frac{967}{5^9}.
\end{aligned}\]
Finally, for~$n=7$,
\[\begin{aligned}
(\hat{d}^1_6)^3 - \PP\big[\bar{\meet}^{\leq 3}_{\CFK(\theta)} \;\big\vert\; \chi^1_{\CFK(\theta)}\!=\!\chi^2_{\CFK(\theta)}\!=\!\move^3 \big] = \frac{1092727}{13824000} - \frac{406417}{5184000} & = \frac{5369}{8294400} 
> \frac{967}{6^9}.
\end{aligned}\]
This completes the proof.
\end{proof}

\section{Proofs Deferred from \texorpdfstring{\Cref{subsec:meeting-time}}{Section 4.2}}

\subsection{Proof of\texorpdfstring{ \Cref{claim:diff-probabilities}}{Lemma 9}}\label{app:claim:diff-probabilities}

    We first express~$d^1_{n-1}$ and~$d^k_{n-2}$ for~$k\in \{0,1,2\}$ in terms of~$d^0_{n-3}$ and~$d^0_{n-4}$, which will be useful to only work with these two terms in what follows.
    By repeatedly applying the formula~$d^{k+1}_m=d^k_m+d^k_{m-1}$ for~$m\in \NN, k\in \{0,\ldots,m-1\}$ from the definition of the difference table, and the well-known recursion~$d^0_m=(m-1)(d^0_{m-1}+d^0_{m-2})$ for~$m\in \NN\setminus \{1\}$, we obtain
    \begin{equation}\label{eqs:fgh}
    \begin{aligned}
        d^1_{n-1} & = d^0_{n-1}+d^0_{n-2} = (n-2)(d^0_{n-2}+d^0_{n-3}) + (n-3)(d^0_{n-3}+d^0_{n-4}) \\
        & = (n-2)(n-3)(d^0_{n-3}+d^0_{n-4}) + (n-2)d^0_{n-3} + (n-3)(d^0_{n-3}+d^0_{n-4}) \\
        & = (n^2-3n+1)d^0_{n-3}+(n-1)(n-3)d^0_{n-4},\\
        d^0_{n-2} & = (n-3)(d^0_{n-3}+d^0_{n-4}),\\
        d^1_{n-2} & = d^0_{n-2}+d^0_{n-3} = (n-3)(d^0_{n-3}+d^0_{n-4}) + d^0_{n-3} = (n-2)d^0_{n-3} + (n-3)d^0_{n-4}, \\
        d^2_{n-2} & = d^1_{n-2}+d^1_{n-3} = (n-2)d^0_{n-3} + (n-3)d^0_{n-4} + d^0_{n-3} + d^0_{n-4} = (n-1)d^0_{n-3}+(n-2)d^0_{n-4}.
    \end{aligned}
    \end{equation}
    
    We now show that~$\Delta_2=-2\Delta_1>0$, with~$\Delta_2=8/81$ if~$n=4$.
    We first analyze~$\Delta_1$ by observing that
    \[
        \PP\big[B_{\AW(\theta),1} \;\big\vert\; \bar{\meet}^{\leq 2}_{\AW(\theta)},\ \chi^1_{\AW(\theta)}\!=\!\chi^2_{\AW(\theta)}\!=\!\move^3\big] = \PP\big[B_{\AW(\theta),1} \;\big\vert\; \chi^1_{\AW(\theta)}\!=\!\chi^2_{\AW(\theta)}\!=\!\move^3\big] = \frac{1}{(n-1)^2},
    \]
    because the event~$B_{\AW(\theta),1}$ is independent from~$\bar{\meet}^{\leq 2}_{\AW(\theta)}$ and each player starts the third permutation at the home location of the other with probability~$1/(n-1)$, and
    \begingroup
    \allowdisplaybreaks
    \[\begin{aligned}
        \PP\big[B_{\CFK(\theta),1} \;\big\vert\; \bar{\meet}^{\leq 2}_\CFK(\theta),\ \chi^1_{\CFK(\theta)}\!=\!\chi^2_{\CFK(\theta)}\!&=\!\move^3\big] = \frac{\PP\big[B_{\CFK(\theta),1},\ \bar{\meet}^{\leq 2}_\CFK(\theta) \;\big\vert\; \chi^1_{\CFK(\theta)}\!=\!\chi^2_{\CFK(\theta)}\!=\!\move^3\big]}{\PP\big[\bar{\meet}^{\leq 2}_\CFK(\theta) \;\big\vert\; \chi^1_{\CFK(\theta)}\!=\!\chi^2_{\CFK(\theta)}\!=\!\move^3\big]} \\
        ={} & \frac{(n-1)!^2\PP\big[B_{\CFK(\theta),1},\ \bar{\meet}^{\leq 2}_\CFK(\theta) \;\big\vert\; \chi^1_{\CFK(\theta)}\!=\!\chi^2_{\CFK(\theta)}\!=\!\move^3\big]}{(d^1_{n-1})^2},
    \end{aligned}\]
    \endgroup
    because~$\CFK(\theta)$ is equivalent to~$\AW(\theta)$ in the first two rounds and thus the non-meeting probability in these rounds, when both players move, is simply~$(\hat{d}^1_{n-1})^2=(d^1_{n-1})^2/(n-1)!^2$.
    We obtain
    \begin{equation}\label{eq:Delta1}
    \begin{aligned}
        \Delta_1 ={} & \PP\big[B_{\AW(\theta),1} \;\big\vert\; \bar{\meet}^{\leq 2}_{\AW(\theta)},\ \chi^1_{\AW(\theta)}\!=\!\chi^2_{\AW(\theta)}\!=\!\move^3\big] - \PP\big[B_{\CFK(\theta),1} \;\big\vert\; \bar{\meet}^{\leq 2}_\CFK(\theta),\ \chi^1_{\CFK(\theta)}\!=\!\chi^2_{\CFK(\theta)}\!=\!\move^3\big]\\
        ={} & \frac{(d^1_{n-1})^2-(n-1)^2(n-1)!^2\PP\big[B_{\CFK(\theta),1},\ \bar{\meet}^{\leq 2}_\CFK(\theta) \;\big\vert\; \chi^1_{\CFK(\theta)}\!=\!\chi^2_{\CFK(\theta)}\!=\!\move^3\big]}{(n-1)^2(d^1_{n-1})^2}.
    \end{aligned}
    \end{equation}
    We now compute (the additive inverse of) the numerator:
    \begingroup
    \allowdisplaybreaks
    \[\begin{aligned}
        (n-1&)^2(n-1)!^2\PP\big[B_{\CFK(\theta),1},\ \bar{\meet}^{\leq 2}_\CFK(\theta) \;\big\vert\; \chi^1_{\CFK(\theta)}\!=\!\chi^2_{\CFK(\theta)}\!=\!\move^3\big] - (d^1_{n-1})^2\\
        ={} & (d^0_{n-2})^2 + 2(n-2)(d^1_{n-2})^2 + \frac{(n-2)(n^2-7n+13)}{n-3} (d^2_{n-2})^2 - (d^1_{n-1})^2\\
        ={} & \frac{n-2}{n-3} \big( (n+1)(d^0_{n-3})^2 +2(n-1)d^0_{n-3}d^0_{n-4} + (n-2)(d^0_{n-4})^2 \big).
    \end{aligned}\]
    \endgroup
    Indeed, the first equality comes from \Cref{lem:cond-probabilities} and the second one from \eqref{eqs:fgh}. 
    Replacing in \eqref{eq:Delta1}, we obtain
    \begin{equation}\label{eq:Delta1-final}
        \Delta_1 = - \frac{n-2}{(n-1)^2(n-3)(d^1_{n-1})^2} \big( (n+1)(d^0_{n-3})^2 +2(n-1)d^0_{n-3}d^0_{n-4} + (n-2)(d^0_{n-4})^2 \big),
    \end{equation}
    which is strictly negative for~$n\geq 4$ and evaluates to
    \(
        - \frac{2}{9\cdot 1\cdot 9} ( 5\cdot 0 + 2\cdot 3\cdot 0\cdot 1 + 2\cdot 1) = -\frac{4}{81}
    \)
    for~$n=4$.

    We next analyze~$\Delta_2$ in a similar way.
    We first note that
    \[\begin{aligned}
        & \PP\big[B_{\AW(\theta),2} \;\big\vert\; \bar{\meet}^{\leq 2}_{\AW(\theta)},\ \chi^1_{\AW(\theta)}\!=\!\chi^2_{\AW(\theta)}\!=\!\move^3\big] = \PP\big[B_{\AW(\theta),2} \;\big\vert\; \chi^1_{\AW(\theta)}\!=\!\chi^2_{\AW(\theta)}\!=\!\move^3\big] = \frac{2(n-2)}{(n-1)^2},
    \end{aligned}\]
    and
    \[\begin{aligned}
        \PP\big[B_{\CFK(\theta),2} \;\big\vert\; \bar{\meet}^{\leq 2}_\CFK(\theta),\ \chi^1_{\CFK(\theta)}\!=\!\chi^2_{\CFK(\theta)}\!& ={}\!\move^3\big]
        ={} \frac{\PP\big[B_{\CFK(\theta),2},\ \bar{\meet}^{\leq 2}_\CFK(\theta) \;\big\vert\; \chi^1_{\CFK(\theta)}\!=\!\chi^2_{\CFK(\theta)}\!=\!\move^3\big]}{\PP\big[\bar{\meet}^{\leq 2}_\CFK(\theta) \;\big\vert\; \chi^1_{\CFK(\theta)}\!=\!\chi^2_{\CFK(\theta)}\!=\!\move^3\big]} \\
        ={} & \frac{(n-1)!^2\PP\big[B_{\CFK(\theta),2},\ \bar{\meet}^{\leq 2}_\CFK(\theta) \;\big\vert\; \chi^1_{\CFK(\theta)}\!=\!\chi^2_{\CFK(\theta)}\!=\!\move^3\big]}{(d^1_{n-1})^2},
    \end{aligned}\]
    where most equalities follow analogously to the case of~$\Delta_1$, but now the probability of~$B_{\AW(\theta),2}$ conditional on moving in the first three rounds is~$2(n-2)/(n-1)^2$, because this event occurs when exactly one player starts the third permutation at the home location of the other.
    We obtain
    \begin{equation}\label{eq:Delta2}
    \begin{aligned}
        \Delta_2 ={} & \PP\big[B_{\AW(\theta),2} \;\big\vert\; \bar{\meet}^{\leq 2}_{\AW(\theta)},\ \chi^1_{\AW(\theta)}\!=\!\chi^2_{\AW(\theta)}\!=\!\move^3\big] - \PP\big[B_{\CFK(\theta),2} \;\big\vert\; \bar{\meet}^{\leq 2}_\CFK(\theta),\ \chi^1_{\CFK(\theta)}\!=\!\chi^2_{\CFK(\theta)}\!=\!\move^3\big]\\
        ={} & \frac{2(n-2)(d^1_{n-1})^2-(n-1)^2(n-1)!^2\PP\big[B_{\CFK(\theta),2},\ \bar{\meet}^{\leq 2}_\CFK(\theta) \;\big\vert\; \chi^1_{\CFK(\theta)}\!=\!\chi^2_{\CFK(\theta)}\!=\!\move^3\big]}{(n-1)^2(d^1_{n-1})^2}.
    \end{aligned}
    \end{equation}
    We now compute the numerator, divided by~$n-2$ for convenience:
    \[\begin{aligned}
        2(d&^1_{n-1})^2-\frac{(n-1)^2(n-1)!^2}{n-2}\PP\big[B_{\CFK(\theta),2},\ \bar{\meet}^{\leq 2}_\CFK(\theta) \;\big\vert\; \chi^1_{\CFK(\theta)}\!=\!\chi^2_{\CFK(\theta)}\!=\!\move^3\big]\\
        = {} & 2\bigg((d^1_{n-1})^2 - 2d^0_{n-2}d^1_{n-2} - (n-3)(d^1_{n-2})^2 - 2(n-3)d^1_{n-2}d^2_{n-2} -\frac{(n-4)(n^2-6n+10)}{n-3}(d^2_{n-2})^2 \bigg) \\
        = {} & \frac{2}{n-3} \big( (n+1)(d^0_{n-3})^2 +2(n-1)d^0_{n-3}d^0_{n-4} + (n-2)(d^0_{n-4})^2 \big).
    \end{aligned}\]
    The first equality comes from \Cref{lem:cond-probabilities} and the second one from \eqref{eqs:fgh}. 
    Replacing in \eqref{eq:Delta2}, we obtain
    \begin{equation*}
        \Delta_2 = \frac{2(n-2)}{(n-1)^2(n-3)(d^1_{n-1})^2} \big( (n+1)(d^0_{n-3})^2 +2(n-1)d^0_{n-3}d^0_{n-4} + (n-2)(d^0_{n-4})^2 \big).
    \end{equation*}
    Since the expression on the right-hand side is strictly positive and minus twice the one on the right-hand side of \eqref{eq:Delta1-final}, this concludes the proof of the (in)equalities regarding~$\Delta_1$ and~$\Delta_2$.

    We now prove that~$\Delta_3>0$ when~$n\geq 5$.
    We first note that
    \[
        \PP\big[B_{\AW(\theta),3} \;\big\vert\; \bar{\meet}^{\leq 2}_{\AW(\theta)},\ \chi^1_{\AW(\theta)}\!=\!\chi^2_{\AW(\theta)}\!=\!\move^3\big] = \PP\big[B_{\AW(\theta),3} \;\big\vert\; \chi^1_{\AW(\theta)}\!=\!\chi^2_{\AW(\theta)}\!=\!\move^3\big] = \frac{(n-2)(n-3)}{(n-1)^2},
    \]
    and
    \[\begin{aligned}
        \PP\big[B_{\CFK(\theta),3} \;\big\vert\; \bar{\meet}^{\leq 2}_\CFK(\theta),\ \chi^1_{\CFK(\theta)}\!=\!\chi^2_{\CFK(\theta)}\!={} &\!\move^3\big]
        ={} \frac{\PP\big[B_{\CFK(\theta),3},\ \bar{\meet}^{\leq 2}_\CFK(\theta) \;\big\vert\; \chi^1_{\CFK(\theta)}\!=\!\chi^2_{\CFK(\theta)}\!=\!\move^3\big]}{\PP\big[\bar{\meet}^{\leq 2}_\CFK(\theta) \;\big\vert\; \chi^1_{\CFK(\theta)}\!=\!\chi^2_{\CFK(\theta)}\!=\!\move^3\big]} \\
        ={} & \frac{(n-1)!^2\PP\big[B_{\CFK(\theta),3},\ \bar{\meet}^{\leq 2}_\CFK(\theta) \;\big\vert\; \chi^1_{\CFK(\theta)}\!=\!\chi^2_{\CFK(\theta)}\!=\!\move^3\big]}{(d^1_{n-1})^2},
    \end{aligned}\]
    where most equalities follow analogously to the previous cases, but now the probability of~$B_{\AW(\theta),3}$ conditional on moving in the first three rounds is~$(n-2)(n-3)/(n-1)^2$, because this event occurs when both players start the third permutation at two different locations, none of which is the home location of the other.
    We obtain
    \begin{equation}\label{eq:Delta3}
    \begin{aligned}
        \Delta_3 ={} & \PP\big[B_{\AW(\theta),3} \;\big\vert\; \bar{\meet}^{\leq 2}_{\AW(\theta)},\ \chi^1_{\AW(\theta)}\!=\!\chi^2_{\AW(\theta)}\!=\!\move^3\big]- \PP\big[B_{\CFK(\theta),3} \;\big\vert\; \bar{\meet}^{\leq 2}_\CFK(\theta),\ \chi^1_{\CFK(\theta)}\!=\!\chi^2_{\CFK(\theta)}\!=\!\move^3\big]\\
        ={} & \frac{(n-2)(n-3)(d^1_{n-1})^2-(n-1)^2(n-1)!^2\PP\big[B_{\CFK(\theta),3},\ \bar{\meet}^{\leq 2}_\CFK(\theta) \;\big\vert\; \chi^1_{\CFK(\theta)}\!=\!\chi^2_{\CFK(\theta)}\!=\!\move^3\big]}{(n-1)^2(d^1_{n-1})^2}.
    \end{aligned}
    \end{equation}
    We now compute the numerator, divided by~$n-2$ for convenience:
    \begingroup
    \allowdisplaybreaks
    \[\begin{aligned}
        (n-3)&(d^1_{n-1})^2-\frac{(n-1)^2(n-1)!^2}{n-2}\PP\big[B_{\CFK(\theta),3},\ \bar{\meet}^{\leq 2}_\CFK(\theta) \;\big\vert\; \chi^1_{\CFK(\theta)}\!=\!\chi^2_{\CFK(\theta)}\!=\!\move^3\big]\\
        ={} & (n-3)(d^1_{n-1})^2 - \frac{2(n^2-7n+13)}{n-3}d^0_{n-2}d^2_{n-2} - 2(n-3)(d^1_{n-2})^2 \\
        & - \frac{4(n-4)(n^2-6n+10)}{n-3} d^1_{n-2}d^2_{n-2} - \frac{n^4-14n^3+75n^2-185n+181}{n-3}(d^2_{n-2})^2\\
        ={} & \frac{(n^3-4n^2+n-2)(d^0_{n-3})^2 + 2n(n-1)(n-4)d^0_{n-3}d^0_{n-4} + (n^3-6n^2+8n-1)(d^0_{n-4})^2}{n-3}.
    \end{aligned}\]
    \endgroup
    As in the previous cases, the first equality comes from \Cref{lem:cond-probabilities} and the second one from \eqref{eqs:fgh}. 
    Replacing in \eqref{eq:Delta3}, we obtain that~$\Delta_3$ equals
    \begin{equation*}
        \frac{(n-2)\big((n^3-4n^2+n-2)(d^0_{n-3})^2 + 2n(n-1)(n-4)d^0_{n-3}d^0_{n-4} + (n^3-6n^2+8n-1)(d^0_{n-4})^2\big)}{(n-1)^2(n-3)(d^1_{n-1})^2}.
    \end{equation*}
    All coefficients in the numerator and all terms in the denominator are strictly positive when~$n\geq 5$.
    This is immediate for the linear factors. For~$n^3-4n^2+n-2$ it follows from the fact that~$n^3>4n^2$ when~$n\geq 5$; for~$n^3-6n^2+8n-1$ it follows from the facts that this expression equals~$14$ when~$n=5$ and that~$n^3\geq 6n^2$ when~$n\geq 6$.
    We conclude that~$\Delta_3>0$, which finishes the proof.\qed

\bibliographystyle{abbrvnat}
\bibliography{bibliography}
	
\end{document}